\documentclass[11pt,letterpaper,reqno]{amsart}
\usepackage[top=1in, bottom=1in, left=1.25in, right=1.25in]{geometry}
\usepackage{amsfonts}
\usepackage{amsmath}
\usepackage{amsthm, amscd}
\usepackage{mathtools}
\usepackage{amssymb}
\usepackage{mathrsfs}
\usepackage{graphicx}
\usepackage{caption}
\usepackage{subcaption}
\usepackage{float}
\usepackage{xypic}
\usepackage[abs]{overpic}
\usepackage[alphabetic]{amsrefs}
\usepackage{etex}
\usepackage{tikz-cd}

\usepackage{hyperref}
\hypersetup{
	colorlinks,
	linkcolor=[rgb]{0,0,0.7},
	urlcolor=[rgb]{0,0,0.4},
	citecolor=[rgb]{0.4,0.1,0}
}

\allowdisplaybreaks

\theoremstyle{plain}\newtheorem{Theorem}{Theorem}[section]
\theoremstyle{plain}\newtheorem{Corollary}[Theorem]{Corollary}
\theoremstyle{plain}\newtheorem{Lemma}[Theorem]{Lemma}
\theoremstyle{plain}\newtheorem{Definition}[Theorem]{Definition}
\theoremstyle{plain}\newtheorem{Proposition}[Theorem]{Proposition}
\theoremstyle{plain}
\theoremstyle{plain}
\theoremstyle{plain}\newtheorem*{Claim*}{Claim}
\theoremstyle{plain}\newtheorem*{Theorem*}{Theorem}
\theoremstyle{plain}\newtheorem*{Lemma*}{Lemma}
\theoremstyle{plain}

\theoremstyle{remark}\newtheorem{remark}[Theorem]{Remark}
\theoremstyle{remark}
\theoremstyle{remark}\newtheorem*{Notation*}{Notation}

\theoremstyle{plain}\newtheorem{alphthm}{Theorem}

\theoremstyle{plain}
\makeatletter
\newtheorem*{rep@theorem}{\rep@title}
\newcommand{\newreptheorem}[2]{
\newenvironment{rep#1}[1]{
 \def\rep@title{#2 \ref{##1}}
 \begin{rep@theorem}}
 {\end{rep@theorem}}}
\makeatother
\newreptheorem{Theorem}{Theorem}
\newreptheorem{Proposition}{Proposition}
\newreptheorem{Corollary}{Corollary}

\numberwithin{equation}{section}

\usepackage{lipsum}

\DeclareMathOperator{\rank}{rank}

\DeclareMathOperator{\Ima}{Im}

\DeclareMathOperator{\SO}{SO}
\DeclareMathOperator{\U}{U}

\DeclareMathOperator{\spann}{span}

\DeclareMathOperator{\id}{id}

\DeclareMathOperator{\lk}{lk}

\DeclareMathOperator{\Diff}{Diff}
\DeclareMathOperator{\Homeo}{Homeo}
\DeclareMathOperator{\Emb}{Emb}
\DeclareMathOperator{\Map}{Map}

\DeclareMathOperator{\inte}{int}
\DeclareMathOperator{\Co}{Co}

\DeclareMathOperator{\orb}{orb}
\DeclareMathOperator{\diff}{diff}
\DeclareMathOperator{\topo}{top}

\newcommand{\bQ}{\mathbb{Q}}

\newcommand{\bZ}{\mathbb{Z}}

\newcommand{\tor}{\mbox{tors}}

\newreptheorem{theorem}{Theorem}

\author{Jianfeng Lin}
\address{Yau Mathematical Sciences Center, Tsinghua University, Beijing 100084, China}
\email{linjian5477@mail.tsinghua.edu.cn}
\author{Yi Xie}
\address{Beijing International Center for Mathematical Research, Peking University, Beijing 100871, China}
\email{yixie@pku.edu.cn}
\author{Boyu Zhang}
\address{Department of Mathematics, The University of Maryland at College Park, Maryland 20742, USA}
\email{bzh@umd.edu}

\title{Pseudo-isotopies of $3$--manifolds with infinite fundamental groups}

\begin{document}

\maketitle
\begin{abstract}
Suppose $Y$ is a compact, connected, oriented $3$--manifold possibly with boundary, such that $\pi_1(Y)$ is infinite. Let $\operatorname{Diff}_\partial(I\times Y)$ denote the group of self-diffeomorphisms of $I\times Y$ that are equal to the identity near the boundary. Let $\operatorname{Diff}_{PI}(I\times Y)$ denote the subgroup of $\operatorname{Diff}_\partial(I\times Y)$ consisting of elements pseudo-isotopic to the identity. Define $\operatorname{Homeo}_\partial(I\times Y)$, $\operatorname{Homeo}_{PI}(I\times Y)$ similarly for homeomorphisms. We show that the canonical map $\pi_0\operatorname{Diff}_{PI}(I\times Y) \to \pi_0\operatorname{Homeo}_{PI}(I\times Y)$ is of infinite rank. As a consequence, $\pi_0\operatorname{Diff}_{PI}(I\times Y)$, $\pi_0\operatorname{Diff}_{\partial}(I\times Y)$, $\pi_0\operatorname{Homeo}_{PI}(I\times Y)$, $\pi_0\operatorname{Homeo}_{\partial}(I\times Y)$ are all abelian groups of infinite rank. 
We also prove that $\pi_0\,C(Y)$ contains an abelian subgroup of infinite rank, and $\pi_0\,C(I\times Y)$ admits a surjection to an abelian group of infinite rank, where $C(X)$ denotes the concordance automorphism group $\Diff(I\times X, \{0\}\times X\cup I\times \partial X)$ or $\Homeo(I\times X, \{0\}\times X\cup I\times \partial X)$. 
These results are proved by studying the actions of barbell diffeomorphisms on the spaces of embedded arcs and configuration spaces. 
\end{abstract}

\section{Introduction}
\label{sec_introduction}
Given an oriented smooth manifold $X$ with boundary, we use $\Diff_{\partial}(X)$ to denote the group of self-diffeomorphisms of $X$ that are the identity near the boundary. We use $\Diff_{PI}(X)$ to denote the subgroup of $\Diff_{\partial}(X)$ consisting of diffeomorphisms that are pseudo-isotopic to the identity. By definition, a diffeomorphism $f$ belongs to $\Diff_{PI}(X)$ if and only if there exists a diffeomorphism $\widetilde{f}:I\times X\to I\times X$ such that $\tilde{f}=\id$ on a neighborhood of $I\times \partial X$, and $\tilde{f}(0,x)=(0,x)$, $\tilde{f}(1,x)=(1,f(x))$ for all $x\in X$.
The group of path-connected components $\pi_{0}\Diff_{\partial}(X)$ is called the \emph{smooth mapping class group} of $X$, and it contains $\pi_{0}\Diff_{PI}(X)$ as a normal subgroup. 

Similarly, when $X$ is an oriented topological manifold, we use $\Homeo_{\partial}(X)$ to denote the topological group of self-homeomorphisms of $X$ that are the identity near the boundary. Let $\Homeo_{PI}(X)$ be the subgroup of $\Homeo_{\partial}(X)$ consisting of homeomorphisms that are (topologically) pseudo-isotopic to the identity. Then the group $\pi_0\Homeo_{\partial}(X)$ is the topological mapping class group of $X$.

Let $I=[0,1]$, let $Y$ be a compact, connected, oriented $3$--manifold. This paper studies the properties of the mapping class groups of $I\times Y$ in the smooth and topological categories. 
Note that $\pi_0\Diff_{PI}(I\times Y)$, $\pi_0\Diff_{\partial}(I\times Y)$, $\pi_0\Homeo_{PI}(I\times Y)$, $\pi_0\Diff_{\partial}(I\times Y)$ are all abelian. This is because one can always isotope two given diffeomorphisms/homeomorphisms on $I\times Y$ such that they have disjoint supports and hence commute. Also note that there is a canonical inclusion $\Diff_{PI}(I\times Y)\to \Homeo_{PI}(I\times Y)$, which induces a homomorphism $\pi_0\Diff_{PI}(I\times Y)\to \pi_0\Homeo_{PI}(I\times Y)$.


Our first main result is stated as follows. 
\begin{alphthm}
    \label{thm_main}
    Suppose $Y$ is a connected, compact, oriented $3$--manifold, possibly with boundary, such that  $\pi_1(Y)$ is infinite. Then the image of the map
    \begin{equation}\label{eqn_main_thm_statement}
\pi_0\Diff_{PI}(I\times Y) \to \pi_0\Homeo_{PI}(I\times Y)
    \end{equation}
    is of infinite rank.
\end{alphthm}

Theorem \ref{thm_main} has the following immediate corollary.

\begin{Corollary}\label{cor: infinite generation}
    Suppose $Y$ is a connected, compact, oriented $3$--manifold, possibly with boundary, such that  $\pi_1(Y)$ is infinite. Then the abelian groups 
    \[\pi_0\Diff_{PI}(I\times Y),\,\,\pi_0\Homeo_{PI}(I\times Y),\,\,\pi_0\Diff_{\partial}(I\times Y),\,\,\pi_0\Homeo_{\partial}(I\times Y)
    \]
 are all of infinite rank.
\end{Corollary}

\begin{remark}
Various special cases of Corollary \ref{cor: infinite generation} were known before:  In the smooth category, Budney-Gabai \cite{budney2019knotted}, and independently by Watanabe \cite{watanabe2020},  proved the case when $Y=S^1\times D^2$. Singh \cite{singh2025pseudo} proved the case when $Y$ is $S^2\times S^1$ (see \cite{FGHK} for an alternative proof)
or the connected sum of two closed aspherical 3-manifolds. Igusa \cite{Igusa2021} proved the case when $Y=(S^2\times S^1)\# Y'$ for a non-simply connected $Y'$. For relevant results in the topological category, see  \cite{budney2025automorphism,Galvin}. See \cite{OhtaWatanabe} for results when $\pi_{1}(Y)$ is finite.
\end{remark}

Gluing $\{0\}\times Y$ with $\{1\}\times Y$ yields homomorphisms
$\Diff_{\partial}(I\times Y)\to \Diff_{\partial}(S^1\times Y)$ and $\Homeo_{\partial}(I\times Y)\to \Homeo_{\partial}(S^1\times Y)$.
It is proved by Igusa \cite{Igusa2021}*{Lemma 5.1}\footnote{The result is originally stated in the smooth category for closed manifolds, but the same argument works in the topological category and for manifolds with boundary.} that these maps induce injections on $\pi_0$.
Therefore, we have the following corollary.

\begin{Corollary}
   Suppose $Y$ is a connected, compact, oriented $3$--manifold, possibly with boundary, such that  $\pi_1(Y)$ is infinite.
   Then the groups
    \[\pi_0\Diff_{PI}(S^1\times Y),\,\,\pi_0\Homeo_{PI}(S^1\times Y),\,\,\pi_0\Diff_{\partial}(S^1\times Y),\,\,\pi_0\Homeo_{\partial}(S^1\times Y)
    \]
contain abelian subgroups of infinite rank.
\end{Corollary}

Our second main result concerns the concordance automorphism groups of $3$-- and $4$--manifolds. Given a smooth manifold $X$, the \emph{concordance diffeomorphism (resp. homeomorphism) group} $C^{\diff}(X)$ (resp. $C^{\topo}$) is defined to be the group of diffeomorphisms (resp. homeomorphisms) on $I\times X$ that are the identity near $(\{0\}\times X)\cup (I\times \partial X)$. The concordance automorphism group is also called the \emph{pseudo-isotopy group}. This is a key object in high-dimensional manifold theory, and has a deep connection with algebraic K-theory. In fact, there is a well-defined stabilization map $C^{\diff}(X)\hookrightarrow C^{\diff}(I\times X)$, which induces an isomorphism on $\pi_{n-1}$ when $\mathrm{dim} X\geq  \max(2n+7,3n+4)$ \cite{Igusa88}. Moreover, the homotopy type of the limiting object $\varinjlim C^{\diff}(I^{k}\times X)$ can be understood via algebraic K-theory \cite{Waldhausen13}, providing a powerful machinery for computing the homotopy groups of $C^{\diff}(X)$ in the stable range. Regarding the group of path components, Cerf \cite{Cerf70} proved that $\pi_{0}\,C^{\diff}(X)=0$ for simply-connected $X$ of dimension $\geq 5$. Hatcher--Wagner \cite{Hatcher73} and Igusa \cite{Igusa84} further computed $\pi_{0}\,C^{\diff}(X)$ for a general $X$ of dimension $\geq 5$. (See \cite[Theorem 3.1]{Hatcher78} for a precise statement.) These results also hold in the topological category \cite{Burghelea82,Pedersen, Hatcher78}.
On the other hand, understanding $\pi_{0}\,C^{\diff}(X)$ and $\pi_{0}\,C^{\topo}(X)$  for $3$-- and $4$-- manifolds remains a challenging question. In fact, apart from a few specific families 
\cite{KS1993,Kwasik,budney2019knotted,Igusa2021,OhtaWatanabe}, it was previously unknown whether $\pi_0\,C^{\diff}(Y)$ or  $\pi_0\,C^{\topo}(Y)$ are nontrivial for a general $3$--manifold $Y$.
Our second main result studies the properties of the concordance automorphism groups of $Y$ and $I\times Y$.
\begin{alphthm}
    \label{thm_main2} Suppose $Y$ is a connected, compact, oriented $3$--manifold, possibly with boundary, such that  $\pi_1(Y)$ is infinite. Let $C=C^{\diff}$ or $C^{\topo}$. Then $\pi_0\,C(Y)$ contains an abelian subgroup of infinite rank, and $\pi_0\,C(I\times Y)$ admits a surjection to an abelian group of infinite rank.
\end{alphthm}

Now we sketch the ideas of the proofs. We start with Theorem \ref{thm_main}.
To construct an infinite collection of linearly independent elements $\varphi_1,\varphi_2,\dots$ in the image of \eqref{eqn_main_thm_statement}, we use Budney--Gabai's barbell diffeomorphisms. The main result is to show that every non-trivial $\mathbb{Z}$--linear combination $\varphi$ of the $\varphi_i$'s is not isotopic to the identity in $\pi_0\Homeo_\partial(I\times Y)$.  To obstruct the isotopy from $\varphi$ to $\id$, we consider the action of $\varphi$ on the space of arcs in $I\times Y$. 

Let $\Emb_\dagger(I,I\times Y)$ be the space of embedded arcs $\gamma:[0,1]\to I\times Y$ such that there exists $y\in Y$, which may depend on $\gamma$, with $\gamma(0) = (0,y)$, $\gamma(1) = (1,y)$, and that the projection image of $\gamma$ to $Y$ is a contractible loop.  In Part \ref{part_space_arcs} of the paper, we will study the properties of $\pi_1\Emb_\dagger(I,I\times Y)$ and $\pi_2\Emb_\dagger(I,I\times Y)$. The group $\pi_1\Emb_\dagger(I,I\times Y)$ will be computed in Proposition \ref{prop_decomposition_pi_1_emb_dagger} using the Dax isomorphism. When $b_1(Y)>0$ and $Y$ is irreducible, we will construct a homomorphism $\Psi$ from $\pi_2\Emb_\dagger(I,I\times Y)$ to $\oplus_\infty\mathbb{Q}$ via configuration spaces of points and the Bousfield--Kan spectral sequence, which generalizes the constructions in \cite{budney2025automorphism,LXZ2025}.

Our strategy is to obstruct the isotopy between $\varphi$ and $\id$ by showing that $\varphi$ acts non-trivially on the set of homotopy classes of curves and surfaces in $\Emb_\dagger(I,I\times Y)$. The homotopy classes of closed curves are given by conjugation classes in $\pi_1\Emb_\dagger(I,I\times Y)$. As for surfaces, we will develop several results in Section \ref{sec_surface_homotopy} that will allow us to obstruct their homotopies using the properties of $\pi_1\Emb_\dagger(I,I\times Y)$ and $\pi_2\Emb_\dagger(I,I\times Y)$.

Another technique that we frequently use is to lift homeomorphisms $\varphi$ on $I\times Y$ to finite covers. If $\widetilde{Y}$ is a finite cover of $Y$ which is Haken, we will prove in Lemma \ref{lem_pull_back_diff_PI} that there is a well-defined pull-back map from the mapping class group of $I\times Y$ to that of $I\times \widetilde{Y}$. Therefore, we may obstruct the isotopy of $\varphi$  by obstructing the isotopy of its pull-back to $I\times \widetilde{Y}$.

We will use different combinations of obstructions for different classes of manifolds $Y$. Let $\hat Y$ be obtained from $Y$ after filling in every boundary component diffeomorphic to $S^2$ with a copy of $D^3$. If $\hat Y$ is irreducible, we will obstruct isotopy by considering the homotopy classes of surfaces in $\Emb_\dagger(I,I\times Y)$, after possibly pulling back to a covering space.  We employ three different arguments for the cases when $\partial \hat Y\neq \emptyset$, when $\hat Y$ is closed and Seifert fibered, and when $\hat Y$ is closed and non-Seifert fibered. This is the content of Part \ref{part_irred_case}.

The case when $\hat{Y}$ is reducible will be addressed in Part \ref{part_red_case}. If $\hat{Y}$ is not diffeomorphic to $S^1\times S^2$ or $\mathbb{RP}^3\# \mathbb{RP}^3$, then the obstruction to isotopy is given by considering homotopy classes of circles in $\Emb_\dagger(I,I\times Y)$. If $\hat{Y}$ is $S^1\times S^2$ or $\mathbb{RP}^3\# \mathbb{RP}^3$, we develop an \emph{ad hoc} argument by comparing the mapping class groups of $I\times Y$ with that of $S^1\times D^3$. Combining all the above arguments concludes the proof of Theorem \ref{thm_main}.

The proof of Theorem \ref{thm_main2} is based on a stronger version of Theorem \ref{thm_main}, which states that the diffeomorphisms we constructed are linearly independent not only in $\pi_0\Homeo_{\partial}(I\times Y)$, but also in $\pi_0\Homeo(I\times Y,I\times \partial Y)$. The stronger forms of Theorem \ref{thm_main} will be stated in Theorems \ref{thm_main_irred_case} and \ref{thm_main_red_case} below. We also note that there are exact sequences
\[
\pi_0\Diff_{\partial}(I\times X)\to \pi_0\,C^{\diff}(X)\to \pi_0\Diff_{PI}(X)\to 0,
\]
\[
\pi_0\Homeo_{\partial}(I\times X)\to \pi_0\,C^{\topo}(X)\to \pi_0\Homeo_{PI}(X)\to 0.
\]
Theorem \ref{thm_main2} will follow as a consequence of the above statements.

\vspace{0.5\baselineskip}
\textbf{Acknowledgments:} 
We would like to thank Alexander Kupers and Yi Liu for helpful correspondence.
J. Lin is partially supported by  National Key R\&D Program of China 2025YFA1017500 and NSFC 12271281. 
Y. Xie is partially supported by NSFC 12341105.
B. Zhang is partially supported by NSF grant DMS-2405271 and a travel grant from the Simons Foundation.

\newpage
\begingroup
\hypersetup{linkcolor=black} 
\tableofcontents
\endgroup

\part{Spaces of arcs in $I\times Y$}
\label{part_space_arcs}
Suppose $Y$ is a connected, oriented, compact $3$--manifold, possibly with boundary. Recall that we denote $I=[0,1]$. 
Part \ref{part_space_arcs} of the paper defines several spaces of arcs in $I\times Y$ and studies the properties of their homotopy groups. 

We will later consider the actions of automorphism groups of $I\times Y$ on the space of arcs, so we introduce the notation for several automorphism groups here. 
If $A\subset X$ are topological spaces, let $\Homeo(X)$ denote the group of self-homeomorphisms of $X$, let $\Homeo(X,A)$ denote the group of elements in $\Homeo(X)$ that fix $A$, and let $\Homeo_0(X)$ denote the group of elements in $\Homeo(X)$ that are homotopic to the identity.
In the following, we will always take $X$ to be a compact manifold, and take $A$ to be a closed domain on the boundary of $X$. Therefore $\Homeo(X), \Homeo_0(X), \Homeo(X,A)$ are topological groups with respect to the $C^0$ topology. 

Let us also recall several definitions from the introduction.
If $X$ is a smooth manifold, let $\Diff_\partial(X)$ denote the group of diffeomorphisms of $X$ that are equal to the identity near the boundary, endowed with the $C^\infty$ topology. Let $\Diff_{PI}(X)$ denote the subgroup of $\Diff_{\partial}(X)$ consisting of diffeomorphisms that are pseudo-isotopic to the identity. By definition, a diffeomorphism $f$ belongs to $\Diff_{PI}(X)$ if and only if there exists a diffeomorphism $\tilde{f}:I\times X\to I\times X$ such that $\tilde{f}=\id$ on a neighborhood of $I\times \partial X$, and $\tilde{f}(0,x)=(0,x), \tilde{f}(1,x)=(1,f(x))$ for all $x\in X$.
We define $\Homeo_\partial(X)$, $\Homeo_{PI}(X)$ similarly for homeomorphisms and endow them with the $C^0$ topology.

If $G$ is one of the automorphism groups above, and $\varphi_1,\varphi_2\in G$ are in the same path-connected component, we say that $\varphi_1$ and $\varphi_2$ are \emph{isotopic} in $G$.

There are various inclusion relations between the automorphism groups. For example, we have injective maps 
\[
\Diff_{PI}(X)\hookrightarrow\Homeo_{PI}(X)\hookrightarrow\Homeo_{0}(X)\hookrightarrow\Homeo(X)
\]
defined by taking an automorphism to itself. We will call these maps and their induced maps on homotopy groups the \emph{canonical maps}.

\section{The embedding spaces of arcs}
\label{sec_embedding_space_definition}

Let $\Emb_\dagger (I,I\times Y)$ be the space of \emph{topological} proper embeddings $\gamma:I\to I\times Y$ such that there exists $y\in Y$, which may depend on $\gamma$, with $\gamma(0) = (0,y)$, $\gamma(1) = (1,y)$, and that the projection of $\gamma$ to $Y$ is a contractible loop.  Define
\begin{equation}
\label{eqn_def_mathscrI}
\begin{split}
\mathscr{I}: Y & \to \Emb_\dagger (I,I\times Y) \\
y & \mapsto (t\mapsto (t,y))
\end{split}
\end{equation}
to be the embedding that takes each $y\in Y$ to the standard arc with image $I\times \{y\}$.

Let $y_0\in Y$ be a base point, and let $\mathscr{I}(y_0)$ be the base point of $\Emb_\dagger(I,I\times Y)$. The base points will be omitted from the notation of homotopy groups when it is clear from the context. 
Define $\Emb_\partial(I,I\times Y)$ to be the set of elements $\gamma\in \Emb_\dagger (I,I\times Y)$ such that $\gamma(0) = (0,y_0)$, $\gamma(1) = (1,y_0)$. Then $\Emb_\partial(I,I\times Y)$ is a subspace of $\Emb_\dagger(I,I\times Y)$, and we have a fibration  $\Emb_\partial(I,I\times Y)\hookrightarrow \Emb_\dagger(I,I\times Y)\to Y$. 
Therefore, if $Y$ is aspherical, then the inclusion induces an isomorphism on homotopy groups
\[
\pi_i \Emb_\partial(I,I\times Y)\cong \pi_i \Emb_\dagger(I,I\times Y)
\]
for all $i\ge 2$. Since the map $\mathscr{I}:Y\to \Emb_\dagger(I,I\times Y)$ is a section for the above
fibration, we have an exact sequence 
$$
0\to \pi_1 \Emb_\partial(I,I\times Y)\to \pi_1 \Emb_\dagger(I,I\times Y) \to \pi_1(Y)\to 0
$$
with a splitting map $\mathscr{I}_\ast:\pi_1(Y)\to \pi_1 \Emb_\dagger(I,I\times Y)$. As a consequence,
we have 
\begin{equation}\label{eq_pi_1_emb_dagger}
\pi_1 \Emb_\dagger(I,I\times Y)\cong \pi_1 \Emb_\partial(I,I\times Y)\rtimes \pi_1(Y).
\end{equation}

Let $\Emb^{sm}_{\partial}(I,I\times Y)\subset \Emb_{\partial}(I,I\times Y)$ denote the subspace consisting of \emph{smoothly} embedded arcs. 
By a theorem of Robinson  \cite{Robinson}\footnote{The original statement of \cite{Robinson} only addressed the embedding spaces of closed manifolds, but the arguments employed there also apply to manifolds with boundary.}, the inclusion map induces an isomorphism on $\pi_1$:
\begin{equation}
\label{eqn_robinson_iso}
\pi_1\Emb^{sm}_{\partial}(I,I\times Y)\cong \pi_1\Emb_{\partial}(I,I\times Y).
\end{equation}
During the preparation of this paper, we learned from Alexander Kupers that the isomorphism \eqref{eqn_robinson_iso} would also follow from an upcoming work of Kremer--Kupers \cite{Kupers}.

Let $\varphi\in \Homeo_{\partial}(I\times Y)$. Then $\varphi$ induces a homeomorphism $\varphi_*:\Emb_{\dagger}(I,I\times Y) \to \Emb_{\dagger}(I,I\times Y)$. It is clear that if $\varphi$ is isotopic to the identity in $\Homeo_\partial(I\times Y)$, then $\varphi_*$ is homotopic to the identity. In fact, we also have the following result, which allows us to obstruct isotopy in $\Homeo(I\times Y, I\times \partial Y)$ using the homotopy class of $\varphi_*$.

\begin{Lemma}
\label{lem_Embdagger_obstruct_isotopy}
    Suppose $\varphi\in \Homeo_{\partial}(I\times Y)$ is isotopic to the identity in $\Homeo(I\times Y, I\times \partial Y)$, and suppose $\varphi$ is homotopic to the identity map of $I\times Y$ relative to $\{0,1\}\times Y$. Then $\varphi_*\circ\mathscr{I}$ and $\mathscr{I}$ are homotopic as maps from $Y$ to $\Emb_{\dagger}(I,I\times Y)$. 
\end{Lemma}

\begin{proof}
    We say that a map $\varphi:Y\times I \to Y\times I$ is \emph{level-preserving}, if it preserves the $I$--coordinate.  Note that there is a fibration (up to homotopy equivalence) 
    \[
    \Homeo_\partial(I\times Y)\hookrightarrow \Homeo(I\times Y, I\times \partial Y) \xrightarrow{p} \Homeo_\partial(\{0,1\}\times Y),
    \]
    where the map $p$ is defined by the restriction to $\{0,1\}\times Y$. So we have an exact sequence
    \[
\pi_1(\Homeo_\partial(\{0,1\}\times Y))\to \pi_0 \Homeo_\partial(I\times Y)\to \pi_0 \Homeo(I\times Y, I\times \partial Y).
    \]
    It is straightforward to see that elements in the image of $\pi_1(\Homeo_\partial(\{0,1\}\times Y))\to \pi_0 \Homeo_\partial(I\times Y)$ are represented by level-preserving homeomorphisms. Therefore, if $\varphi$ is isotopic to the identity in $\Homeo(I\times Y, I\times \partial Y)$, then after an isotopy in $\Homeo_\partial(I\times Y)$, we may assume that $\varphi$ is a level-preserving homeomorphism. 

   Let $F:I\times (I\times Y)\to I\times Y$ be a homotopy from $\varphi$ to $\id$ relative to $\{0,1\}\times Y$. Define $\hat F:I\times (I\times Y)\to I\times Y$ by $\hat F(t,(s,y)) = (s,\pi_Y(F(t,(s,y))))$, where $\pi_Y:I\times Y\to Y$ is the projection. Then $\hat F$ is a homotopy from $\varphi$ to $\id$ relative to $\{0,1\}\times Y$ via level-preserving maps.

   Let $G:I\times Y \to \Emb_{\dagger}(I,I\times Y)$ be the map that takes $(t,y)\in I\times Y$ to the arc 
   \begin{align*}
       I & \to I\times Y \\
       s &\mapsto \hat{F}(t,(s,y)).
   \end{align*}
   Note that the arc above is always embedded, because it preserves the $I$--coordinate.  Therefore, $G$ is well-defined. It is straightforward to verify that $G$ is a homotopy between $\mathscr{I}$ and $\varphi_*\circ \mathscr{I}$. 
\end{proof}

\section{The homotopy groups of $\Emb_{\dagger}(I,I\times Y)$}
\label{sec_embedding_space_homotopy}
In the following, we use $\pi_i^\bQ$ to denote the rational homotopy groups. 
The purpose of this section is to study the properties of $\pi_1\Emb_\dagger(I,I\times Y)$ and $\pi_2^\bQ\Emb_\dagger(I,I\times Y)$.

\subsection{The Dax isomorphism}
\label{subsec_Dax_iso}
In this subsection, we compute $\pi_1\Emb_\dagger(I,I\times Y)$.
To simplify the computation, we assume $\pi_1(Y)$ is infinite throughout Section \ref{subsec_Dax_iso}.  This assumption implies that the universal cover of $Y$ is non-compact, so it is homotopy equivalent to a simply connected 2-dimensional CW complex, which is always homotopy equivalent to a wedge sum of 2-spheres. As a consequence, we have 
\begin{equation}
\label{eqn_pi3Y_generated_by_pi2Y}
\pi_3^\bQ(Y)=[\pi_2^\bQ(Y),\pi_2^\bQ(Y)],
\end{equation}
 where $[\cdot,\cdot]$ denotes the Whitehead product.

Given two elements $\gamma_0,\gamma_1\in \Emb_{\partial}(I,I\times Y)$, we may view 
$\gamma_1$ as a path in $[1,2]\times Y$. Then we can concatenate the two paths to obtain a path
$\gamma_0\ast\gamma_1$ in $[0,2]\times Y\cong I\times Y$. It is straightforward to check that
this construction makes $\Emb_{\partial}(I,I\times Y)$ into an H-space. As a result, 
$\pi_1\Emb_{\partial}(I,I\times Y)$ is an abelian group. 

Let $\Map_\partial (I, I\times Y)$ denote the space of continuous maps $\gamma:I\to I\times Y$ such that $\gamma(0)=(0,y_0)$, $\gamma(1)=(1,y_0)$, where $y_0\in Y$ is the base point.
We have a canonical embedding 
$$\Emb_{\partial}(I,I\times Y)\to \Map_\partial (I, I\times Y)$$ which induces 
a homomorphism
$$
\mathcal{F}: \pi_1\Emb_{\partial}(I,I\times Y)\to \pi_1\Map_\partial (I, I\times Y) \cong\pi_2(Y).
$$
A right inverse $\mathcal{J}:\pi_2(Y)\to \pi_1\Emb_{\partial}(I,I\times Y)$ for $\mathcal{F}$ can be 
constructed as follows: given an element $[\eta]\in \pi_2(Y)$ represented by
$
\eta: (I^2,\partial I^2)\to (Y,y_0)
$
we define $\mathcal{J}([\eta])=[\xi]$ by
\begin{align*}
\xi: (I,\partial I)&\to  \Emb_{\partial}(I,I\times Y;\mathscr{I}(y_0))\\
t&\mapsto (s\mapsto (s,\eta(t,s))).
\end{align*}
Therefore, we have
\begin{equation}\label{eq_pi_1_emb_partial}
\pi_1\Emb_{\partial}(I,I\times Y)\cong \pi_1^D \Emb_{\partial}(I,I\times Y)\oplus \pi_2(Y)
\end{equation}
where we use $\pi_1^D \Emb_{\partial}(I,I\times Y)$ to denote the kernel of $\mathcal{F}$.
By the Dax isomorphism theorem (see \cite{gabai2021self}*{Theorem 0.3}) and \eqref{eqn_robinson_iso}, we have 
\begin{equation*}\label{eq_Dax_iso}
\pi_1^D \Emb_{\partial}(I,I\times Y)\cong \bZ[\pi_1(Y)\setminus\{1\}]/\Ima d_3
\end{equation*}
where the homomorphism $d_3: \pi_3(I\times Y)\to \bZ[\pi_1(Y)\setminus\{1\}]$ is defined
in \cite{gabai2021self}*{Definition 3.8}. 

 According to  \cite{KT24}*{Proposition 3.14}, for $a,b\in \pi_2(I\times Y)$, we have
$$
d_3([a,b])=\lambda(a,b)+\lambda(b,a)
$$
where $\lambda(a,b)$ denotes the equivariant intersection number (omitting the coefficient of 1)
of $a,b\in \pi_2(I\times Y)$.
Since $\pi_2(I\times Y)\cong \pi_2(Y)$, we may assume $a$ and $b$ are represented by immersed spheres 
lying in different $Y$-slices of $I\times Y$. Then we have $\lambda(a,b)=\lambda(b,a)=0$ since 
there is no intersection point. By \eqref{eqn_pi3Y_generated_by_pi2Y}, we conclude that $ d_3=0$, and hence
\begin{equation}\label{eq_dax_iso_no_d3}
\pi_1^D \Emb_{\partial}(I,I\times Y)\cong \bZ[\pi_1(Y)\setminus\{1\}].
\end{equation}

Now we prove the main result of this subsection.
\begin{Proposition}\label{prop_decomposition_pi_1_emb_dagger}
We have a semidirect product decomposition
\begin{equation}
\label{pi_1_Emb_daggar_semi}
\pi_1 \Emb_\dagger(I,I\times Y)\cong (\bZ[\pi_1(Y)\setminus\{1\}] \oplus \pi_2(Y))
\rtimes \pi_1(Y), 
\end{equation}
where $\mathscr{I}_*$ is the embedding map of $\pi_1(Y)$ in $\pi_1 \Emb_\dagger(I,I\times Y)$ as a semidirect product component, and the conjugation action of $\pi_1(Y)$ on $\bZ[\pi_1(Y)\setminus\{1\}] \oplus \pi_2(Y)$ is given by the conjugation actions on the generators on 
$\mathbb{Z}[\pi_1(Y)\setminus \{1\}]$ and the 
 standard action on $\pi_2(Y)$.
\end{Proposition}
\begin{proof}
By \eqref{eq_pi_1_emb_dagger}, \eqref{eq_pi_1_emb_partial} 
and \eqref{eq_dax_iso_no_d3},
\begin{align*}
\pi_1 \Emb_\dagger(I,I\times Y) &\cong \pi_1 \Emb_\partial(I,I\times Y) \rtimes \pi_1(Y) \\
& \cong (\bZ[\pi_1(Y)\setminus\{1\}] \oplus \pi_2(Y))
\rtimes \pi_1(Y).
\end{align*}
We need to 
 compute the conjugation action of $\mathscr{I}_*(\pi_1(Y))$ on $\pi_1\Emb_\partial(I,I\times Y)$, viewed as subgroups of $\pi_1\Emb_\dagger(I,I\times Y)$. 
 
  For each $[\gamma]\in \pi_1(Y)$, let $H:Y\times I\to Y$ be a smooth isotopy of $Y$ from $\id$ to $\id$ such that the trajectory of the base point is  $\gamma$.  
For each $\alpha\in \pi_1\Emb_\partial(I,I\times Y)$, let $\hat\alpha:S^1\to \Emb_\partial(I,I\times Y)$ be a representative of $\alpha$. For each $t$, the map $H(-,t)$ induces a self-homeomorphism of $\Emb_\dagger(I,I\times Y)$, and we use $\hat \alpha_t$ to denote the image of $\hat \alpha$ under this homeomorphism. 
Then $\hat\alpha_t$ $(t\in[0,1])$ defines a 1-parameter family of maps from $S^1$ to $\Emb_\dagger(I,I\times Y)$ such that the trajectory of the base point is $\mathscr{I}\circ \gamma$. As a result, $\hat\alpha_1$ represents the image of the conjugation action of $[\mathscr{I}\circ\gamma]$ on $\alpha$. Since the Dax isomorphism \eqref{eq_dax_iso_no_d3} is natural with respect to self-diffeomorphisms of $I\times Y$, and the induced map of $H(-,1)$ on $\pi_1(Y)$ is the conjugation by $[\gamma]$, we conclude that, under the identification via \eqref{eq_dax_iso_no_d3}, the conjugation action of $\pi_1(Y)$ on $\pi_1^D\Emb_\partial(I,I\times Y)$ is given by conjugations on the generators of $\mathbb{Z}[\pi_1(Y)\setminus \{1\}]$.
Similarly, the action of $H(-,1)$ on $\pi_2(Y)$
is the standard action by $[\gamma]$ and it is 
clear from the construction of $\mathcal{J}$ that
it commutes
with the action of $H(-,1)$.
\end{proof}

\subsection{From embedding spaces to configuration spaces}
\label{subsec_emb_to_conf}
For the rest of Section \ref{sec_embedding_space_homotopy}, we assume that $Y$ is irreducible and $b_1(Y)>0$. This implies that $Y$ is aspherical. 
The rest of Section \ref{sec_embedding_space_homotopy} constructs a homomorphism $\Psi$ from the homotopy group $\pi_2^\bQ\Emb_\dagger(I,I\times Y)$ to an infinite-dimensional $\bQ$--vector space, and proves several basic properties of the map $\Psi$. 

We start by establishing a relationship between embedding spaces and configuration spaces of points. 
The constructions in this subsection are extensions of \cite[Section 4]{budney2025automorphism} and \cite[Section 7.4]{LXZ2025}.
Let $Y$ be as above. Let $\widetilde{Y}$ be the normal covering of $Y$ such that the image of $\pi_1(\widetilde{Y})\to \pi_1(Y)$ equals the kernel of the quotient map
\[
 \tau:\pi_1(Y)\to H_1(Y;\mathbb{Z})/\tor,
\]
where $H_1(Y;\mathbb{Z})/\tor$ denotes the quotient of $H_1(Y;\mathbb{Z})$ by its torsion.
 
For each $a\in H_1(Y;\mathbb{Z})/\tor$, let $t^a:\widetilde{Y}\to \widetilde{Y}$ denote the deck transformation determined by $a$. We also abuse notation and use $t^a$ to denote the map $\id_I \times t^a: I\times \widetilde{Y} \to I\times \widetilde{Y}$. Let $\tilde{y}_0$ be the base point on $\widetilde{Y}$, which is a lifting of the base point $y_0\in Y$.  

\begin{remark}
    All the constructions in Sections \ref{subsec_emb_to_conf} - \ref{subsec_hom_from_pi5C3} still work if we replace $\tau$ with an arbitrary surjection from $\pi_1(Y)$ to an abelian group. The reason for this specific choice of $\tau$ is the following: (1) The naturality property stated by Proposition \ref{prop_naturality_Psi} requires that the codomain of $\tau$ be functorial with respect to $Y$ and the map $\tau$ be natural. (2) In the computation of $r\circ \Psi(\alpha)$ in the proof of Theorem \ref{thm_main_irred_case_when_xi1_xi2_exist_in_pi1}, we need to use the fact that the codomain of $\tau$ has no torsion. 
\end{remark}

\begin{Definition}
   For every positive integer $n$, define 
\begin{align*}
C^\tau_n(I\times Y)=&\{(x_0,x_1,\cdots,x_n,x_{n+1})\in (I\times \widetilde{Y})^{n+2} | 
x_0 = (0,\tilde{y}_0), x_{n+1} = (1,\tilde{y}_0), \\
&x_i\neq t^a x_j ~\text{for all}~a\in (H_1(Y;\mathbb{Z})/\tor)\setminus\{0\}, \\
&\quad 0\le i<j\le n+1 \}
\end{align*}

\end{Definition}
The sequence of spaces $C^\tau_n(I\times Y)$ $(n\ge 0)$ can be made into a cosimplicial space where 
the coface map $\delta_n^i: C^\tau_n(M) \to C^\tau_{n+1}(M)$ ($0\le i \le n+1$) is defined by doubling the $i^{th}$
point
and the codegeneracy map $\sigma_n^i: 
C^\tau_n(M) \to C^\tau_{n-1}(M)$ ($1\le i \le n$) is defined by forgetting the $i^{th}$ point. Notice that
the coface maps are cofibrations since they are embeddings of closed submanifolds. 

Following \cite[Definition 3.1]{LXZ2025}, we say that a topological space $X$ is endowed with a \emph{$\Delta^n$--structure}, if for $S\subset\{0,\dots,n\}$, there is an associated subset $X_S$ such that $X_S\cap X_{S'}= X_{S\cap S'}$. 
If $X_1, X_2$ are endowed with $\Delta^n$--structures, we use $\Map_{\Delta}(X_1,X_2)$ to denote the space of continuous maps from $X_1$ to $X_2$ that takes $(X_1)_S$ to $(X_2)_S$ for all $S$.

The space $C_n^\tau(I\times Y)$ is endowed with a $\Delta^n$--structure using its cosimplicial structure:
 each $S\subset \{0,\dots,n\}$ induces an embedding $C^\tau_{|S|-1}(I\times Y)\to C^\tau_n(I\times Y)$ from
the cosimplicial structure. We define $C_{n,S}^\tau(I\times Y)$ to be the image of this map.
We also identify the simplex $\Delta^n$ with the set $\{(t_0,t_1, t_2,\dots ,t_n,t_{n+1})|0= t_0\le t_1\le t_2\le\dots \le t_n\le t_{n+1} = 1\}$, which is endowed with a standard $\Delta^n$--structure with $\Delta^n_S = \{(t_0,t_1, t_2,\dots ,t_n,t_{n+1})\in\Delta^n| t_i = t_{i+1} ~\text{for all}~ i\notin S\}$.

Then for each $n$, we have a map\footnote{This map was denoted by $\Psi_n$ in \cite{LXZ2025}.} 
\[
\Phi_n: \Emb_\partial (I,I\times Y) \to \Map_\Delta(\Delta^n,C_n^\tau(I\times Y))
\]
defined by taking $f\in \Emb_\partial (I,I\times Y)$ to the restriction of $\tilde{f}^n:I^n \to (I\times \widetilde{Y})^n$ to $\Delta^n$,
where $\tilde f:I\to I\times \widetilde{Y}$ is the lifting of $f\in \Emb_\partial (I,I\times Y)$ such that $\tilde{f}(0)=(0,\tilde{y}_0)$.

Consider the above map when $n=3$, and consider its induced map on $\pi_2^\bQ$. We obtain a homomorphism 
\[
(\Phi_3)_*: \pi_2^\bQ \Emb_\partial (I,I\times Y) \to \pi_2^\bQ \Map_\Delta(\Delta^n,C_3^\tau(I\times Y))
\]

\subsection{The Bousfield--Kan spectral sequence}
As discussed in \cite[Section 7.4]{LXZ2025}, there is a special case of the Bousfield--Kan spectral sequence that computes $\pi_2^\bQ \Map_\Delta(\Delta^3,C_3^\tau(I\times Y))$. The spectral sequence implies that 
$\pi_2^\bQ \Map_\Delta(\Delta^3,C_3^\tau(I\times Y))$ is isomorphic to the quotient of $\pi_5^\bQ(C_3^\tau(I\times Y))$ by the images of the coface maps  
$$
\delta_2^i:\pi_5^\bQ(C_2^\tau(I\times Y))\to \pi_5^\bQ C_3^\tau(I\times Y),~0\le i \le 3
$$
and the image of a subgroup of $\oplus^6 \pi_4^\bQ C_1^\tau(I\times Y)$. 
Here, to simplify notation, we use $\delta_2^i$ to also denote the induced maps by $\delta_2^i$ on homotopy groups. 
Under our assumptions, $\pi_4^\bQ(C_1^\tau(I\times Y))\cong 0$ since $Y$ is aspherical.  Therefore, we have
\begin{equation}\label{eqn_Bousfield-Kan}
 \pi_2^\bQ \Map_\Delta(\Delta^3,C_3^\tau(I\times Y))\cong \pi_5^\bQ C_3^\tau(I\times Y)/ \sum_i \Ima \delta_2^i.
\end{equation}

The base points for the homotopy groups in \eqref{eqn_Bousfield-Kan} are given in \cite[Section 3.3]{LXZ2025}; we briefly review the definitions here. In general, a homotopy group is well-defined as long as we choose the base point in a simply connected set. The base point for the homotopy group $\pi_5^\bQ(C_3^\tau(I\times Y))$ is chosen from the set of $(x_0,x_1,x_2,x_3,x_4)\in C_3^\tau(I\times Y)$ such that $x_i=(t_i,\tilde{y}_0)$ and $0=t_0\le t_1\le\dots\le t_3\le t_4=1$. The base point for $\pi_2^\bQ \Map_\Delta(\Delta^3,C_3^\tau(I\times Y))$ is taken to be $\Phi_3(\mathscr{I}(y_0))$. Recall that $\mathscr{I}$ is defined by \eqref{eqn_def_mathscrI}.

Now we develop some tools to compute $\Ima(\delta_2^i)$. 
Let $\hat Y$ be the universal covering of $Y$. Then the universal covering of $C_n^\tau(I\times Y)$ can be 
described as 
\begin{align*}
\hat C^\tau_n(I\times Y)=&\{(x_0,x_1,\cdots,x_n,x_{n+1})\in (I\times \hat{Y})^{n} | x_0 = (0,\tilde{y}_0), x_{n+1}=(1,\tilde{y}_0),\\
&x_i\neq t^\alpha x_j ~\text{for all}~\alpha\in \pi_1(Y)\setminus \ker \tau, 
 0\le i<j\le n+1 \}
\end{align*}
where $t^\alpha$ denotes the deck transformation determined by $\alpha\in \pi_1(Y)$.

We define the following elements in the homotopy groups of $ C^\tau_n(I\times Y)$ 
(cf. \cite[Definitions 4.3, 4.8]{LXZ2025}).
\begin{Definition}
\label{def_ti^alpha_wij}
For $1\le i < j \le n$ and $\alpha \in \pi_1(Y)$, we define 
$$
t_i^\alpha w_{ij}\in \pi_3 \hat C_n^\tau(I\times Y)
$$  
by letting $x_i$ rotate around $t^\alpha x_j$ in a small coordinate neighborhood.   
\end{Definition}

Note that if $\alpha \in \ker \tau$, then $t_i^\alpha w_{ij}=0 \in \pi_3 \hat C_n^\tau(I\times Y)$. So we will usually only invoke Definition \ref{def_ti^alpha_wij} when $\alpha \in  \pi_1(Y) \setminus \ker \tau$.

Since $\pi_3 \hat C_n^\tau(I\times Y)\cong\pi_3C_n^\tau(I\times Y)$, we also view $t_i^\alpha w_{ij}$
as an element in $\pi_3C_n^\tau(I\times Y)$ by abuse of notation. 
We have the following formulas by the definition of coface maps.
\begin{equation}
\label{eqn_delta_2^i_on_w}
\begin{split}
\delta_2^0 (t_1^\alpha w_{12})&=t_2^\alpha w_{23} \\
\delta_2^1 (t_1^\alpha w_{12})&=t_1^\alpha w_{13}+t_2^\alpha w_{23} \\
\delta_2^2 (t_1^\alpha w_{12})&=t_1^\alpha w_{12}+t_1^\alpha w_{13} \\
\delta_2^3 (t_1^\alpha w_{12})&=t_1^\alpha w_{12}.\\
\end{split}
\end{equation}

The following proposition is proved by a similar argument as \cite[Proposition 4.26, Lemma 7.8(a)]{LXZ2025}.
\begin{Proposition}\label{prop_pi_3_C5_wij}
The vector space $\pi_5^\bQ  C_2^\tau(I\times Y)$ is generated by 
$$
[t_1^\alpha w_{12}, t_1^\beta w_{12}] ~\text{for}~\alpha,\beta\in \pi_1(Y)\setminus \ker \tau, \alpha\neq \beta
$$
where $[\cdot,\cdot]$ denotes the Whitehead product.
\end{Proposition}
\begin{proof}
Let $\inte(\hat C^\tau_n(I\times Y))$ denote the set of points $(x_0,\dots,x_{n+1})\in \hat C^\tau_n(I\times Y)$ such that each $x_i$ is in the interior of $(I\times \hat Y)^n$. Then the inclusion of $\inte(\hat C^\tau_n(I\times Y))$ in $\hat C^\tau_n(I\times Y)$ is a homotopy equivalence, and we have a fibration 
    \[
    \inte(I\times \hat Y) \setminus \{t^\alpha p|\alpha \in \pi_1(Y)\setminus \ker \tau\} \to \inte(\hat C^\tau_2(I\times Y))\to \inte(I\times \hat Y),
    \]
    where $p\in \inte(I\times \hat Y)$ is a fixed point. Since $Y$ is aspherical, we have $\pi_5^\bQ \inte(I\times \hat Y) \setminus \{t^\alpha p|\alpha \in \pi_1(Y)\setminus \ker \tau\} \cong \pi_5^\bQ \inte(\hat C^\tau_n(I\times Y))$. Note that $\inte(I\times \hat Y) \setminus \{t^\alpha p|\alpha \in \pi_1(Y)\setminus \ker \tau\}$ is homotopy equivalent to the wedge sum of a collection of $S^3$'s, where each $S^3$ corresponds to an element of $\{t^\alpha p|\alpha \in \pi_1(Y)\setminus \ker \tau\}$. So its $\pi_5^\bQ$ is generated by Whitehead products on $\pi_3^\bQ$, and a straightforward diagram chasing yields the desired result. 
\end{proof}

Proposition \ref{prop_pi_3_C5_wij} and \eqref{eqn_delta_2^i_on_w} allow us to compute $\Ima(\delta_2^i)$ directly.

\subsection{A homomorphism from $\pi_5C_3^\tau(I\times Y)$}
\label{subsec_hom_from_pi5C3}
Fix an orientation on $S^5$. Suppose $\mu:S^5\to C_3^\tau(I\times Y)$ is a continuous map. 
For 
$0\neq a\in H_1(Y;\bZ)/\tor$,
let $\Co^i_j(a,\mu)$ denote the set of $p\in S^5$ such that 
$t^a(x_i)$ has the same projection to $Y$ as $x_j$ and has a greater $I$--coordinate than $x_j$, where $(x_1,x_2,x_3)=\mu(p)$. After a generic perturbation of $\mu$, the set $\Co_i^j(a,\mu)$ is a transversely cut codimension 2 closed submanifold of $S^5$, and it admits an induced orientation.

Let $a,b\in H_1(Y;\mathbb{Z})/\tor$ be such that $a,b,a-b\neq 0$. The next lemma follows verbatim as \cite[Lemma 7.1]{LXZ2025}.

\begin{Lemma}
    For generic $\mu$, the linking number    
    \[
    \lk\big(\Co^1_2(a,\mu)-\Co^1_3( a-b,\mu),\Co^3_1(b-a,\mu)- \Co^3_2(b,\mu)\big)
    \]
    is well-defined and only depends on the homotopy class of $\mu$. \qed
\end{Lemma}
Define $\theta_{-a,-b}: \pi_5(C_3^\tau(I\times Y))\to \mathbb{Z}$ to be the above linking number.\footnote{We add negative signs to the subscripts here so that Proposition \ref{prop_Theta_wij} will have simpler signs.}

Similarly, 
suppose $a,b \in H_1(Y;\mathbb{Z})/\tor$, $a,b,a-b\neq 0$ and $1\le i< j\le 3$.  
Then $\Co^i_j(a,\mu)\cap \Co^i_j(b,\mu)=\emptyset$ for all $\mu$. Therefore, for generic $\mu$, the linking number
\[\lk\big(\Co^i_j(a,\mu),\Co^i_j(b,\mu)\big)\] is well-defined and
only depends on the homotopy class of $\mu$. We define $\zeta^{i,j}_{-a,-b}(\mu)$ to be this linking number. Note that $\zeta^{i,j}_{a,b}(\mu)=-\zeta^{i,j}_{b,a}(\mu)$ for all $i,j,a,b$.

It is clear from the definition that $\theta_{a,b}$ and $\zeta^{i,j}_{a,b}$ give rise to group homomorphisms from $\pi_5C_3^\tau(I\times Y)$ to $\mathbb{Z}$. 
We abuse notation and also use $\theta_{a,b}$ and  $\zeta^{i,j}_{a,b}$ to denote the maps in $\bQ$--coefficients. Define $\Theta: \pi_5^\bQ C_3^\tau(I\times Y)\to \oplus_\infty \bQ$ to be the direct sum of $\theta_{a,b}$'s 
and $\zeta^{i,j}_{a,b}$'s. It is straightforward to verify the following result.

\begin{Proposition}\label{prop_Theta_wij}
The following identities hold on $\pi_5^\bQ C_3^\tau(I\times Y)$. 
\begin{align*}
  \theta_{a,b}([ t_i^\alpha w_{ij}, t_i^\beta w_{ij}])&=0 \\
  \theta_{a,b}([ t_1^{\alpha} w_{12}, t_2^{\beta} w_{23}]) &=
\begin{cases}
\pm 1~~~\text{if}~a=\tau(\alpha),b=-\tau(\beta)\\
0~~~\text{otherwise}
\end{cases}  \\
 \zeta^{i,j}_{a,b}([ t_1^{\alpha} w_{12}, t_2^{\beta} w_{23}])&=0\\
  \zeta^{i,j}_{a,b}([t_i^{\alpha} w_{ij}, t_i^{\beta} w_{ij}])&=
  \begin{cases}
\pm 1~~~\text{if}~a=\tau(\alpha),b=\tau(\beta)\\
\mp 1~~~\text{if}~a=\tau(\beta),b=\tau(\alpha)\\
0~~~\text{otherwise}
\end{cases} \\
\zeta^{i,j}_{a,b}([t_k^{\alpha} w_{kl}, t_k^{\beta} w_{kl}])&= 0~~~\text{if}~(i,j)\neq (k,l).
\end{align*}
\qed
\end{Proposition}
We may also choose the orientation conventions so that the values $\pm 1$ in Proposition \ref{prop_Theta_wij} are all equal to $+1$, and the values $\mp 1$ are equal to $-1$.

We can now compute the image of $\Theta\circ \delta_2^i$ in $\oplus_\infty \bQ$ by \eqref{eqn_delta_2^i_on_w}, Proposition \ref{prop_pi_3_C5_wij} and Proposition \ref{prop_Theta_wij}. 
In the case that $Y=S^1\times D^2$,
this has been done in \cite[Section 3]{budney2019knotted}. In the following, let $\hat\theta_{a,b}$, $\hat{\zeta}_{a,b}^{i,j}$ $(a,b,a-b\neq 0, 1\le i<j\le 3)$ denote the basis vectors of $\oplus_\infty \mathbb{Q}$ that correspond to $\theta_{a,b}$ and $\zeta_{a,b}^{i,j}$ respectively. Then the image of $\Theta$ is contained in the space spanned by $\hat\theta_{a,b}$ and $\hat \zeta_{a,b}^{i,j}-\hat \zeta_{b,a}^{i,j}$.

\begin{Corollary}
    $\Ima(\Theta\circ \delta_2^i)$ is described as follows.
    \begin{enumerate}
        \item $\Ima(\Theta\circ \delta_2^0)$ is generated by vectors of the form $\hat\zeta_{a,b}^{2,3}-\hat\zeta_{b,a}^{2,3}$.
        \item $\Ima(\Theta\circ \delta_2^1)$ is generated by vectors of the form
        \[
        (\hat\zeta_{a,b}^{1,3} - \hat\zeta_{b,a}^{1,3}) + (\hat\zeta_{a,b}^{2,3}-\hat\zeta_{b,a}^{2,3})-\hat\theta_{a-b,-b}+\hat\theta_{b-a,-a}.
        \]
        \item $\Ima(\Theta\circ \delta_2^2)$ is generated by vectors of the form 
        \[
(\hat\zeta_{a,b}^{1,2} - \hat\zeta_{b,a}^{1,2}) + (\hat\zeta_{a,b}^{1,3} - \hat\zeta_{b,a}^{1,3}) - \hat\theta_{a,a-b} + \hat\theta_{b,b-a}.
        \]
        \item $\Ima(\Theta\circ \delta_2^3)$ is generated by vectors of the form $\hat\zeta_{a,b}^{1,2}-\hat\zeta_{b,a}^{1,2}$.
    \end{enumerate}
\end{Corollary}

\begin{proof}
    Statements (1), (2) follow from Proposition \ref{prop_Theta_wij} and \eqref{eqn_delta_2^i_on_w}. Statements (3), (4) follows from \cite[Eq. (4.3)]{LXZ2025} and a similar computation. 
\end{proof}

Note that $\Ima(\Theta)\subset \spann\{\hat \theta_{a,b},\hat\zeta_{b,a}^{i,j}-\hat\zeta_{b,a}^{j,i}\}$. By the above corollary, we have 
\begin{align*}
&\spann\{\hat \theta_{a,b},\hat\zeta_{b,a}^{i,j}-\hat\zeta_{b,a}^{j,i}\}/\sum_i\Ima(\delta_2^i)\\
\cong\,& \spann\{\hat\theta_{a,b}\}/\spann\{\hat\theta_{a-b,-b}-\hat\theta_{b-a,-a}- \hat\theta_{a,a-b} + \hat\theta_{b,b-a}\},
\end{align*}
where the isomorphism takes the equivalence class of $\hat\theta_{a,b}$ to the equivalence class of $\hat\theta_{a,b}$. 
Now we can define the desired map $\Psi$.
\begin{Definition}\label{Dfn_Psi}
Define $\Psi$ to be the composition of
\begin{align*}
&\pi_2^\bQ \Emb_\partial (I,I\times Y)\xrightarrow{(\Phi_3)_*} \pi_2^\bQ \Map_\Delta(\Delta^n,C_n^\tau(I\times Y))\\
\cong\, & \pi_5^\bQ C_3^\tau(I\times Y)/\sum_i \Ima(\delta_2^i)\xrightarrow{\Theta} \spann\{\hat \theta_{a,b},\hat\zeta_{b,a}^{i,j}-\hat\zeta_{b,a}^{j,i}\}/\sum_i\Ima(\delta_2^i)
\\
\cong\, & \spann\{\hat\theta_{a,b}\}/\spann\{\hat\theta_{a-b,-b}-\hat\theta_{b-a,-a}- \hat\theta_{a,a-b} + \hat\theta_{b,b-a}\}.
\end{align*}
\end{Definition}

\subsection{Properties of $\Psi$}
\label{subsec_property_Psi}
    Given $Y$ as above, denote 
    \[
    \mathcal{W}(Y) = \spann\{\hat\theta_{a,b}\}/\spann\{\hat\theta_{a-b,-b}-\hat\theta_{b-a,-a}- \hat\theta_{a,a-b} + \hat\theta_{b,b-a}\},
    \]
    where $(a,b)$ goes over all pairs of elements in $H_1(Y;\mathbb{Z})/\tor$ such that $a,b,a-b\neq 0$. Then $\mathcal{W}(Y)$ is the codomain of $\Psi$. If $Y'$ is another manifold satisfying the same assumptions as $Y$ and $i: Y'\hookrightarrow Y$ is a codimension-0 orientation-preserving embedding that induces an injection on $H_1(-;\mathbb{Z})/\tor$, then there is an induced map $i_*:\mathcal{W}(Y')\to \mathcal{W}(Y)$. Note that since $i_*$ induces an injective map on $H_1(-;\mathbb{Z})/\tor$, whenever 
    $a,b\in H_1(Y';\mathbb{Z})/\tor$
    satisfy $a,b,a-b\neq 0$, their images in $H_1(Y;\mathbb{Z})/\tor$ satisfy the same condition. 

\begin{Proposition}
\label{prop_embedding_and_codomain_of_Psi}
    Let $i:Y'\hookrightarrow Y$ be as above. 
    Let $r:\mathcal{W}(Y)\to \mathcal{W}(Y')$ be defined by 
    \[
    r(\hat \theta_{a,b}) = \begin{cases}
        \hat\theta_{(i_*)^{-1}(a),(i_*)^{-1}(b)} \quad & \text{if } a,b\in \Ima(i_*:H_1(Y';\mathbb{Z})/\tor\to H_1(Y;\mathbb{Z})/\tor)\\
        0 & \text{otherwise}.
    \end{cases}
    \]
    Then $r$ is well-defined, and $r\circ i_* = \id$ on $\mathcal{W}(Y)$. In particular, $i_*:\mathcal{W}(Y')\to \mathcal{W}(Y)$ is injective.
\end{Proposition}
\begin{proof}
Given a 3-manifold $Z$, let 
\begin{align*}
V(Z)&=\spann\{\hat\theta_{a,b}|a,b\in H_1(Z;\bZ)/\tor,\,a,b,a-b\neq 0\} \\
R(Z)&=\spann\{\hat\theta_{a-b,-b}-\hat\theta_{b-a,-a}- \hat\theta_{a,a-b} + \hat\theta_{b,b-a}|a,b\in H_1(Z;\bZ)/\tor,\,
a,b,a-b\neq 0\}.
\end{align*}
It is clear that $\mathcal W(Z)=V(Z)/R(Z)$. The map
$i:Y'\hookrightarrow Y$ induces a map $\hat{i}: V(Y')\to V(Y)$ which has a left inverse
$\hat{r}: V(Y)\to V(Y')$ defined by the same formula in the definition of $r$. 
In order to prove the proposition, it suffices to show that $\hat r (R(Y))\subset R(Y')$. 
If $a,b\in \Ima i_\ast$, then it is clear 
$$
\hat r(\hat\theta_{a-b,-b}-\hat\theta_{b-a,-a}- \hat\theta_{a,a-b} + \hat\theta_{b,b-a})\in R(Y').
$$
If either $a$ or $b$ is not in $\Ima i_\ast$, then we have
$$
\hat r(\hat\theta_{a-b,-b}-\hat\theta_{b-a,-a}- \hat\theta_{a,a-b} + \hat\theta_{b,b-a})=0\in R(Y').
$$
So the result is proved.
\end{proof}

It is known by the work of Budney--Gabai in \cite[Section 4]{budney2025automorphism} that $\mathcal{W}(Y)$ is infinite dimensional if $Y=S^1\times D^2$. Since we assume $b_1(Y)>0$, Proposition \ref{prop_embedding_and_codomain_of_Psi} implies that $\mathcal{W}(Y)$ is infinite dimensional.
    
\begin{Proposition}
\label{prop_naturality_Psi}
 Let $i:Y'\hookrightarrow Y$ be as above, and suppose that $i$ takes the base point of $Y'$ to the base point of $Y$.
We have a commutative diagram
\[
\begin{tikzcd}
\pi_2^\bQ\Emb_\dagger(I,I\times Y')  \arrow[r,"\Psi'"]\arrow[d]& 
\mathcal{W}(Y') \arrow[d,"i_\ast"]\\
 \pi_2^\bQ \Emb_\dagger (I,I\times Y) \arrow[r,"\Psi"]  & \mathcal{W}(Y)
\end{tikzcd}
\]
\end{Proposition}
\begin{proof}
The embedding $i$ lifts to a map $\tilde i: \widetilde Y' \to  \widetilde Y$ which satisfies 
$$
\tilde i(t^a\cdot x)= t^{i_\ast (a)}\cdot i(x) 
$$
for all $a\in H_1(Y';\bZ)/\tor$ and $x\in \widetilde Y'$. 
By abuse of notation, we use $\tilde i$ to also denote the map $\id_I \times \tilde i: I\times \widetilde{Y}'\to I\times \widetilde{Y}$. 
Since 
\[
i_\ast:H_1(Y';\bZ)/\tor\to H_1(Y; \bZ)/\tor
\]
is injective, 
$\tilde i$ induces a well-defined map  $C_n^\tau(Y')\to C_n^\tau(Y)$, which in turn induces
a map $\Map_\Delta(\Delta^n,C_n^\tau(I\times Y'))\to \Map_\Delta(\Delta^n,C_n^\tau(I\times Y))$.
Now the proposition follows from the construction of $\Theta$ and the fact that $\tilde i$ is $I$-fiber-preserving.
\end{proof}

Now we show that the homomorphism $\Psi$ in Definition \ref{Dfn_Psi} is invariant under an action of $\pi_1(Y)$. 
For $\beta\in \pi_i(X,x_0)$ and $u\in \pi_1(X,x_0)$, we use $\beta^u$ to denote the $\pi_1$ action of $u$ on $\beta$. Recall that we fix a base point $y_0\in Y$, and let $\mathscr{I}(y_0)$ be the base point of $\Emb_\dagger(I,I\times Y)$.
\begin{Proposition}
\label{prop_invariance_Psi_under_pi1_Y}
 Suppose $\beta\in \pi_2^\bQ\Emb_\dagger(I,I\times Y)$ and $u\in \Ima(\mathscr{I}_*:\pi_1(Y)\to \pi_1\Emb_\dagger(I,I\times Y))$. Then $\Psi(\beta) = \Psi(\beta^u)$.
\end{Proposition}

Proposition \ref{prop_invariance_Psi_under_pi1_Y} will follow as a special case of the next proposition.
Suppose $y_0,y_1\in Y$ are two choices of base points on $Y$, we compare the $\Psi$ maps defined with respect to $y_0$ and $y_1$. To differentiate the notation, we temporarily use $\Psi_{y_i}:\pi_2(\Emb_\dagger(I,I\times Y),\mathscr{I}(y_i))\to\mathcal{W}(Y)$ $(i=0,1)$ denote the $\Psi$ map defined with respect to the base point $y_i$. Note that $\mathcal{W}(Y)$ does not depend on the base point. Let $u:[0,1]\to Y$ be an arc from $y_0$ to $y_1$ in $Y$. For $\beta\in \pi_2(\Emb_\dagger(I,I\times Y),\mathscr{I}(y_0))$, let $\beta^u\in \pi_2(\Emb_\dagger(I,I\times Y),\mathscr{I}(y_1))$ denote the image of $\beta$ under the base-point-change isomorphism induced by $\mathscr{I}\circ u$. Then we have the following result.

\begin{Proposition}
\label{prop_compare_Psi_all_base_point_Y}
Let $y_0,y_1,\beta,u$ be as above. Then
 $\Psi_{y_1}(\beta^u) = \Psi_{y_0}(\beta)$.
\end{Proposition}

\begin{proof}
After homotopy, assume that $u$ is smooth.
There is an isotopy $F:[0,1]\times Y\to Y$ that equals the identity near the boundary such that $u(t) = F(t,y_0)$, $F(0,-)=\id$. 

Let $i_t=F(1,t):Y\to Y$, let $\hat{i}_t:\Emb_\dagger(I,I\times Y)\to\Emb_\dagger(I,I\times Y)$ be the induced homeomorphism on $\Emb_\dagger(I,I\times Y)$, and let $\hat \beta:S^2\to \Emb_\dagger(I,I\times Y)$ be a representative of $\beta$. Then $\hat{i}_t\circ \hat \beta$ $(t\in[0,1])$ is a homotopy such that the trajectory of the base point is $\mathscr{I}\circ u$. As a result, $\hat{i}_1\circ \hat \beta$ is a representative of $\beta^u$. By Proposition \ref{prop_naturality_Psi}, we have  $\Psi_{y_1}(\beta^u) = (i_1)_*\circ \Psi_{y_0}(\beta)$.
Since $i_1$ induces the identity on homology, the desired result is proved.
\end{proof}

\part{The case when $Y$ is irreducible after filling $D^3$'s}
\label{part_irred_case}
Let $\hat Y$ be obtained by filling each boundary component of $Y$ that is diffeomorphic to $S^2$ with a copy of $D^3$. In Part \ref{part_irred_case} of this paper, we prove Theorem \ref{thm_main} in the case when $\hat Y$ is irreducible. 

For the convenience of notation, if $i:M_1\to M_2$ is a codimension-zero embedding of compact manifolds, and if $\varphi \in \Homeo_\partial(M_1)$, we use $i_*(\varphi)$ to denote the homeomorphism of $M_2$ obtained by extending $i\circ\varphi \circ i^{-1}$ from $\Ima(i)$ to $M_2$ by the identity. We call $i_*(\varphi)$ the \emph{extension} of $\varphi$ from $M_1$ to $M_2$ via the embedding $i$. If $i$ is smooth, then $i_*$ takes a diffeomorphism to a diffeomorphism. 

The main result of Part \ref{part_irred_case} is the following theorem. 

\begin{alphthm}
\label{thm_main_irred_case}
    Suppose $Y$ is a compact, connected, irreducible, oriented $3$--manifold, such that $\pi_1(Y)$ is infinite. 
    Then there exists a smooth embedding $i:Y_0=S^1\times D^2\to Y$, such that the image of the composition map 
    \[
        \pi_0 \Diff_{PI}(I\times Y_0)\xrightarrow{i_*} \pi_0 \Diff_{PI}(I\times Y) \to \pi_0\Homeo (I\times Y,I\times \partial Y)
    \]
is of infinite rank.
\end{alphthm}

Note that since $\pi_0 \Diff_{PI}(I\times Y_0)$ is abelian, its image is also abelian, so it makes sense to refer to the rank of the image. 

\begin{Lemma*}
 Theorem \ref{thm_main_irred_case} implies Theorem \ref{thm_main} in the case when $\hat Y$ is irreducible.
\end{Lemma*}
\begin{proof}
 Filling in boundary components diffeomorphic to $S^2$ with copies of $D^3$ does not change the fundamental group, so $\hat Y$ satisfies the assumptions of Theorem \ref{thm_main_irred_case}. 
Let $i:Y_0\to \hat{Y}$ be the embedding given by Theorem \ref{thm_main_irred_case}. After isotopy, we may assume that the image of $i$ is contained in $Y$. Let $j:Y\to\hat Y$ be the embedding of $Y$ in $\hat Y$. Then we have a commutative diagram
\begin{equation}
\label{eqn_com_diagram_main_thm_implication}
\begin{tikzcd}
    \pi_0 \Diff_{PI}(I\times Y_0)\arrow[r,"i_*"] & \pi_0 \Diff_{PI}(I\times Y) \arrow[r,"h_1"] \arrow[d,"j_*"] & \pi_0\Homeo (I\times Y,I\times \partial Y) \arrow[d,"j_*"] \\
   & \pi_0 \Diff_{PI}(I\times \hat{Y}) \arrow[r,"h_2"]  & \pi_0\Homeo (I\times \hat{Y},I\times \partial{\hat Y}),
\end{tikzcd}
\end{equation}
where $h_1$, $h_2$ are the canonical maps that takes $[\varphi]$ to $[\varphi]$.
By Theorem \ref{thm_main_irred_case}, the composition $ \pi_0 \Diff_{PI}(I\times Y_0)\to \pi_0\Homeo (I\times \hat{Y},I\times \partial{\hat Y})$ is of infinite rank, so $h_1$ is also of infinite rank. Note that the composition map 
\[
\pi_0\Diff_{PI}(I\times Y) \to \pi_0\Homeo_{PI}(I\times Y) \to \pi_0\Homeo(I\times Y,I\times \partial Y)
\]
is equal to $h_1$, therefore the map $\pi_0\Diff_{PI}(I\times Y) \to \pi_0\Homeo_{PI}(I\times Y)$ is of infinite rank.
\end{proof}

\begin{Lemma*}
 Theorem \ref{thm_main_irred_case} implies Theorem \ref{thm_main2} in the case when $\hat Y$ is irreducible.
\end{Lemma*}
\begin{proof}
There is a fibration
\[
\Diff_{\partial}(I\times X)\hookrightarrow C^{\diff}(X)\xrightarrow{ev} \Diff_{PI}(X)\to 0,
\]
where $ev_*([f])=[f|_{\{1\}\times X}]$, which induces an exact sequence
\[
\pi_0\Diff_{\partial}(I\times X)\to \pi_0\,C^{\diff}(X)\xrightarrow{ev_*} \pi_0\Diff_{PI}(X)\to 0.
\]
Similarly, we have an exact sequence
\[
\pi_0\Homeo_{\partial}(I\times X)\to \pi_0\,C^{\topo}(X)\xrightarrow{ev_*} \pi_0\Homeo_{PI}(X)\to 0.
\]

When $X=I\times Y$,  the desired statement follows from the results that $\pi_0\Diff_{PI}(I\times Y)$ and $\pi_0\Homeo_{PI}(I\times Y)$ are abelian groups of infinite rank.

When $X=Y$, the composition 
\[
\pi_0\Diff_\partial(I\times Y)\to \pi_0\,C^{\diff}(Y)\to \pi_0\,C^{\topo}(Y) \to \pi_0 \Homeo(I\times Y,I\times \partial Y)
\]
equals the map $h_1$ in \eqref{eqn_com_diagram_main_thm_implication}, so its image is an abelian group of infinite rank. As a result, the images of $\pi_0\Diff_\partial(I\times Y)\to \pi_0\,C^{\diff}(Y)$ and $\pi_0\Diff_\partial(I\times Y)\to \pi_0\,C^{\topo}(Y)$ are abelian groups of infinite rank.
\end{proof}

The rest of Part \ref{part_irred_case} is devoted to the proof of Theorem \ref{thm_main_irred_case}.

\section{Homotopy classes of maps from a surface}
\label{sec_surface_homotopy}

We will study $\Homeo(I\times Y, I\times \partial Y)$ by considering its actions on surfaces in $\Emb_\dagger(I,I\times Y)$. 
This section develops several general results about surfaces in topological spaces.
Let $\Sigma$ be a compact, connected, oriented surface, possibly with boundary, and let $X$ be a topological space. We study the properties of the homotopy classes of maps from $\Sigma$ to $X$ relative to $\partial \Sigma$. Later, we will take $X$ to be $\Emb_\dagger(I,I\times Y)$.

\subsection{Interior connected sum with $\pi_2(X)$}
Suppose $f:\Sigma\to X$ is a map and $p\in\inte(\Sigma)$, and let $\alpha\in \pi_2(X,f(p))$. We define an \emph{interior connected sum} of $f$ with $\alpha$ at $p$, which will be denoted by $f\#_p\alpha$. This is a map from $\Sigma$ to $X$, and is well-defined up to homotopy relative to $A$ for an arbitrary given compact set $A\subset \Sigma\setminus\{p\}$. 

Here is the definition of $f\#_p\alpha$.
We first define a map $\xi_p:\Sigma\to \Sigma\vee_p S^2$ as follows. Let $D_p$ be a small disk on $\Sigma$ centered at $p$, let $\xi_p$ be the quotient map of $\Sigma$ defined by collapsing $\partial D_p$ to a point, identifying the quotient with $\Sigma \vee_p S^2$. We also require that the homeomorphism from $D_p/\partial D_p$ to $S^2$ is orientation-preserving, and that the map $\xi_p$ equals the identity outside of a neighborhood of $D_p$.
Let $\alpha\in \pi_2(X,f(p))$, and let $\hat \alpha:S^2\to X$ be a representative of $\alpha$. Define $f\# _p\alpha$ to be the composition
\[
\Sigma \xrightarrow{\xi_p} \Sigma \vee S^2\xrightarrow{f\vee \hat \alpha} X. 
\]
The map $f\#_p \alpha$ depends on the choices of $\hat \alpha$ and $\xi_p$. If $A\subset \Sigma\setminus\{p\}$ is a fixed compact set, by requiring $D_p$ to be disjoint from $A$, the map $f\#_p \alpha$ is well-defined up to homotopy relative to $A$.

The rest of this section compares the homotopy classes of $f$ and $f\#_p\alpha$. 

\subsection{Interior connected sum on surfaces with boundary}
In the following, if $x_0\in X$ is a base point, $u\in \pi_1(X,x_0)$, $\beta\in \pi_i(X,x_0)$, we use $\beta^u\in \pi_i(X,x_0)$ to denote the $\pi_1$--action of $u$ on $\beta$. 

\begin{Lemma}
\label{lem_surface_hmtpy_rel_boundary}
    Suppose $\Sigma$ is a connected, oriented, compact surface and $\partial \Sigma\neq \emptyset$. Let $:\Sigma\to X$ be a map, let $p\in\inte(\Sigma)$, and $\alpha\in \pi_2(X,f(p))$. Then $f$ is homotopic to $f\#_p\alpha$ relative to $\partial \Sigma$ if and only if $\alpha$ is in the subgroup of $\pi_2(X,f(p))$ spanned by all elements of the form
    \[
    \beta - \beta^u,
    \]
    where $\beta\in \pi_2(X,f(p))$ and $u$ is in the image of $f_*: \pi_1(\Sigma,p)\to \pi_1(X,f(p))$.
\end{Lemma}

\begin{proof}
    Endow $\Sigma$ with a CW complex structure relative to $\partial \Sigma$ such that there is exactly one 2-cell and no 0-cell. Let $\Sigma^{(1)}$ be the 1-skeleton; namely, $\Sigma^{(1)}$ equals the union of $\partial \Sigma$ and the 1-cells. Let $e:D^2\to \Sigma$ be the characteristic map of the 2-cell, where $D^2$ denotes the 2-dimensional \emph{closed} disk. We also require that $p$ is in (the interior of) the 2-cell. 

    Let $\Map_{\partial \Sigma}(\Sigma,X)$ be the space of all maps $g:\Sigma \to X$ such that $g|_{\partial \Sigma} = f|_{\partial \Sigma}$, endowed with the compact-open topology. Similarly, let $\Map_{\partial \Sigma}(\Sigma^{(1)},X)$ be the set of all maps $g:\Sigma^{(1)} \to X$ such that $g|_{\partial \Sigma} = f|_{\partial \Sigma}$. Then the restriction map 
    \[
    \Map_{\partial \Sigma}(\Sigma,X) \to \Map_{\partial \Sigma}(\Sigma^{(1)},X)
    \]
    is a fibration, and the fiber $F$ is canonically identified with the set of all maps $g:D^2\to X$ such that $g|_{\partial D^2} = f\circ e|_{\partial D^2}$. 

    Consider the map from $\pi_2(X,f(p))$ to $\pi_0(F)$ that takes $\alpha\in \pi_2(X,f(p))$ to the homotopy class of $(f\circ e)\#_{q}\alpha$, where $q = (e)^{-1}(p)$. Then this map is a bijection, and we use it to identify $\pi_0(F)$ with $\pi_2(X,f(p))$.
    
    Let the base points of $\Map_{\partial \Sigma}(\Sigma,X)$ and $\Map_{\partial \Sigma}(\Sigma^{(1)},X)$ be given by $f$ and $f|_{\Sigma^{(1)}}$. We have an exact sequence 
\[
\pi_1 \Map_{\partial \Sigma}(\Sigma^{(1)},X) \xrightarrow{\partial_*} \pi_0(F) \to \pi_0\Map_{\partial\Sigma}(\Sigma,X).
\]
Note that $\pi_0(F)$ admits an abelian group structure via the isomorphism $\pi_0(F)\cong \pi_2(X,f(p))$, and the map $\partial_*$ from $\pi_1 \Map_{\partial \Sigma}(\Sigma^{(1)},X)$ to $\pi_0(F)$ is a group homomorphism. The exactness of the above sequence implies that two elements of $\pi_0(F)$ map to the same element in $\pi_0\Map_{\partial\Sigma}(\Sigma,X)$ if and only if their difference is in the image of $\partial_*$. 

As a consequence, $f$ is homotopic to $f\#_p\alpha$ relative to $\partial \Sigma$ if and only if $\alpha\in \pi_2(X,f(p))$ is in the image of 
\[
\partial_*:\pi_1 \Map_{\partial \Sigma}(\Sigma^{(1)},X) \to \pi_0(F)\cong \pi_2(X,f(p)).
\]

Now we compute the image of $\partial_*: \pi_1 \Map_{\partial \Sigma}(\Sigma^{(1)},X) \to\pi_2(X,f(p))$. 
Let $J$ be the set of 1-cells of $\Sigma^{(1)}$. Then 
\[
\pi_1\Map_{\partial \Sigma}(\Sigma^{(1)},X)\cong \oplus_{j\in J}\, \pi_2(X, f(p)).
\]
For each $j\in J$, let $\pi_2^{(j)}(X, f(p))$ denote the component in the above direct sum decomposition that corresponds to $j$. 
Then $\partial_*(\pi_2^{(j)}(X, f(p)))$ is described as follows. 
 Let $q$ be a point in the interior of the 1-cell, then $e^{-1}(q)$ contains two points on $\partial D^2$; write them as $\hat q_1,\hat q_2$. Let $\gamma_i$ be an arc in $D^2$ from $p$ to $\hat q_i$ $(i=1,2)$. The concatenation of $f(\gamma_1)$ and the reverse of $f(\gamma_2)$ in $X$ defines an element $u_j$ in $\pi_1(X,f(p))$. Note that $u_j$ only depends on $j$ and does not depend on the other choices in the construction. By a straightforward computation, we have
 \[
\partial_*(\pi_2^{(j)}(X, f(p))) = \{\beta-\beta^{u_j}\mid \beta\in \pi_2(X, f(p))\}.
 \]
 Since elements of the form $u_j$ generate the image of $\pi_1(\Sigma,p)$ in $\pi_1(X,f(p))$, the lemma is proved. 
\end{proof}

\subsection{Interior connected sum on closed surfaces}
\label{subsec_interior_connected_sum_homotopy}
This subsection proves several technical lemmas on the homotopy classes of maps from \emph{closed} surfaces. 
The first lemma considers the homotopy classes relative to a base point. 
\begin{Lemma}
\label{lem_closed_surface_basepoint_connected_sum_homotopy}
Let $\Sigma$ be a connected, oriented, closed surface, let $p,q$ be distinct points on $\Sigma$. Let $f:\Sigma\to X$ be a map, and let $\alpha\in \pi_2(X,f(p))$. View $f\#_p\alpha$ as a map defined up to homotopy relative to $q$. 
Then $f\#_p \alpha$ is homotopic to $f$ relative to $q$ if and only if $\alpha$ is contained in the subgroup of $\pi_2(X,f(p))$ generated by all elements of the form
    \[
    \beta - \beta^u,
    \]
    where $\beta\in \pi_2(X,f(p))$ and $u$ in the image of $f_*: \pi_1(\Sigma,p)\to \pi_1(X,f(p))$. 
\end{Lemma}

\begin{proof}
    Let $D$ be a small closed disk in $\Sigma$ around $q$ that is disjoint from $p$, and let $\mathring\Sigma=\Sigma\setminus \inte(D)$. Then $\mathring\Sigma/\partial \mathring\Sigma\cong \Sigma$. Let $r: \mathring\Sigma \to \Sigma$ be a quotient map that equals the identity near $p$. Then $f$ and $f\#_p\alpha$ are homotopic relative to $q$ if and only if $r\circ f$ and $r\circ (f\#_p\alpha) = (r\circ f)\#_p \alpha$ are homotopic relative to $\partial \mathring\Sigma$. So the result is proved by applying Lemma \ref{lem_surface_hmtpy_rel_boundary} to $\mathring\Sigma$.
\end{proof}

The next two lemmas show that the existence of free homotopies implies the existence of homotopies relative to a base point, under additional technical assumptions. 

\begin{Lemma}
\label{lem_surface_hmtpy_closed}
Suppose $\Sigma$ is a connected, oriented, closed surface, $f:\Sigma\to X$ is a map, $p\in \Sigma$, and $\alpha\in \pi_2(X,f(p))$. Let $q\in \Sigma\setminus\{p\}$, and view $f\#_p\alpha$ as a map defined up to homotopy relative to $q$. 
Suppose further that the centralizer of the image of $\pi_1(\Sigma)$ in $\pi_1(X)$ is trivial. Then $f$ and $f\#_p \alpha$ are homotopic if and only if they are homotopic relative to $q$.
\end{Lemma}

\begin{proof}
The ``if'' part is obvious; we prove the ``only if'' part of the lemma. Suppose $f$ and $f\#_p \alpha$ are homotopic to each other. Choose a CW complex structure of $\Sigma$ such that $q$ is the unique 0-cell and $p$ is in a 2-cell. 
Let $f'$ be a representative of $f\#_p \alpha$ such that $f=f'$ on the 1-skeleton. Take a homotopy from $f$ to $f'$, let $\gamma$ be the trajectory of $q$ under the homotopy. Then $\gamma$ defines an element $[\gamma]\in \pi_1(X,f(q))$. Since $f$ and $f\#_p \alpha$ restrict to the same map on the 1-cells, the element $[\gamma]$ must be commutative with all elements in the image of $f_*: \pi_1(\Sigma,q)\to \pi_1(X,f(q))$. By the assumptions, we deduce that $[\gamma]$ is trivial. By the homotopy extension theorem, there exists a homotopy from $f$ to $f\#_p \alpha$ that preserves the image of $q$.
\end{proof}

\begin{Lemma}
\label{lem_surface_hmtpy_torus}
Suppose $\Sigma=S^1\times S^1$. Let $f:\Sigma\to X$ be a map, let $p$ be a point on $\Sigma$, and let $\alpha\in \pi_2(X,f(p))$. 
Let $q\in \Sigma\setminus\{p\}$, and view $f\#_p\alpha$ as a map defined up to homotopy relative to $q$. 
Assume further that $f$ is equal to a composition map
\begin{equation}
\label{eqn_factor_map_from_torus_through_circle}
\Sigma \cong S^1\times S^1 \xrightarrow{\text{projection}} S^1 \xrightarrow{g} X.
\end{equation}
Then $f$ and $f\#_p \alpha$ are homotopic if and only if they are homotopic relative to $q$.
\end{Lemma}

\begin{proof}
The ``if'' part is obvious; we prove the ``only if'' part of the lemma.
Take the standard CW complex structure with one $0$--cell and one $1$--cell on $S^1$, and endow $\Sigma\cong S^1\times S^1$ with the product CW complex structure. 
We require that $q$ is the unique $0$--cell, and $p$ is contained in the $2$--cell. 
Let $\hat q$ be the projection image of $q\in \Sigma$ in $S^1$, then $\hat q$ is the unique $0$--cell of $S^1$. 
Let $f'$ be a representative of $f\#_p\alpha$ such that $f=f'$ on the $1$--skeleton of $\Sigma$. Suppose $f$ is homotopic to $f'$.

Let $H:[0,1]\times \Sigma\to X$ be a homotopy such that $H(0,-) = f$, $H(1,-)= f'$. Let $\gamma:[0,1]\to X$ be defined by $\gamma(t) = H(t,q)$. Let $g$ be the map in \eqref{eqn_factor_map_from_torus_through_circle}. Then both $\gamma$ and $g$ are loops in $X$ based at $f(q)$, and we have 
\[
[g] = [\gamma]\cdot [g]\cdot [\gamma]^{-1} \in \pi_1(X,f(q)).
\]
Hence, there exists a free homotopy from $g$ to itself such that the trajectory of $\hat q$ is the reverse of $\gamma$. Since $f$ factors through $g$, this implies that there exists a homotopy $G:[0,1]\times \Sigma\to X$ such that $G(0,-) = G(1,-) = f$ and the trajectory of $q$ is the reverse of $\gamma$. 

The composition of the homotopies $G$ and $H$ then gives a homotopy from $f$ to $f'$ such that the trajectory of $q$ is a contractible loop in $X$. Therefore, by the homotopy extension theorem, there exists a homotopy from $f$ to $f'$ that preserves the image of $q$.
\end{proof}

We will later apply Lemma \ref{lem_surface_hmtpy_torus} to the case when there is a natural $\U(1)$ action on $X$. This will be the case, for example, when $Y$ is a $\U(1)$ bundle and $X=\Emb_\dagger(I,I\times Y)$. 

Suppose $m:\U(1)\times X\to X$ is a $\U(1)$ action on $X$. Identify $S^1$ with $\U(1)$, and let $f:S^1\times S^1\to X$ be a map. Define a map $m\circledast f:S^1\times S^1\to X$ by $m\circledast f(e^{ix},e^{iy}) = m(e^{-ix}, f(e^{ix},e^{iy}))$. It is clear that the operation $m\circledast -$ defines a bijection on the homotopy classes of maps from $S^1\times S^1$ to $X$. We record the following two lemmas for later reference.

\begin{Lemma}
\label{lem_m_circledast_alpha_connected_sum}
    Let $p\in \{1\}\times S^1\subset S^1\times S^1$. Let $\alpha\in \pi_2(X,f(p))$. Then $m\circledast(f\#_p\alpha)$ is homotopic to $(m\circledast f)\# \alpha$. 
\end{Lemma}
\begin{proof}
    Let $v:\U(1)\to \U(1)$ be a smooth map with degree $1$ such that $v(e^{ix})=1$ when $x$ is close to $0$. Define $m'\circledast f:S^1\times S^1\to X$ by 
    \[
    m'\circledast f(e^{ix},e^{iy}) = m(v(e^{-ix}), f(e^{ix},e^{iy})).
    \]
    Since $v$ is homotopic to $\id$ relative to $1\in S^1$, we know that $m'\circledast f$ is homotopic to $m\circledast f$ relative to $\{1\}\times S^1$. On the other hand, $m'\circledast(f\#_p\alpha)$ and $(m'\circledast f)\# \alpha$ are represented by the same map. So the lemma is proved. 
\end{proof}

Lemmas \ref{lem_surface_hmtpy_torus}, \ref{lem_m_circledast_alpha_connected_sum}, and \ref{lem_closed_surface_basepoint_connected_sum_homotopy} imply the following result.

\begin{Lemma}
\label{lem_map_from_torus_degenerate_after_m}
Suppose $\Sigma=S^1\times S^1$. Let $f:\Sigma\to X$ be a map, let $p$ be a point on $\Sigma$, and let $\alpha\in \pi_2(X,f(p))$. 
Assume there exists a $\U(1)$ action $m:\U(1)\times X\to X$ such that $m\circledast f$ is equal to a composition map
\[
\Sigma \cong S^1\times S^1 \xrightarrow{\text{projection}} S^1 \xrightarrow{g} X.
\]
Then $f$ and $f\#_p \alpha$ are homotopic if and only if $\alpha$ is contained in the subgroup of $\pi_2(X,f(p))$ generated by all elements of the form
    \[
    \beta - \beta^u,
    \]
    where $\beta\in \pi_2(X,f(p))$ and $u$ in the image of $(m\circledast f)_*: \pi_1(\Sigma,p)\to \pi_1(X,(m\circledast f)(p))$. 
\end{Lemma}

\begin{proof}
    Since $m\circledast -$ defines a bijection on the homotopy classes of maps from $\Sigma$ to $X$, we know that $f$ is homotopic to $f\#_p \alpha$ if and only if $m\circledast f$ is homotopic to $m\circledast (f\#_p \alpha)$. Note that rotating the coordinates of the $S^1$--factors of $\Sigma$ changes the definition of $m\circledast f$ but does not change its homotopy class. After rotating the coordinates of $\Sigma$, we may assume $p\in \{1\}\times S^1\subset S^1\times S^1$.  By Lemma \ref{lem_m_circledast_alpha_connected_sum}, we know that $m\circledast (f\#_p \alpha)$ is homotopic to $(m\circledast f)\#_p \alpha$. By the assumptions and Lemma \ref{lem_surface_hmtpy_torus}, $m\circledast f$ is homotopic to $(m\circledast f)\#_p \alpha$ if and only if they are homotopic relative to a base point $q\in \Sigma\setminus\{p\}$. Hence, the desired result follows from Lemma \ref{lem_closed_surface_basepoint_connected_sum_homotopy}. 
\end{proof}

\section{Diffeomorphisms on $S^1\times D^3$ and the map $\Psi$}
\label{sec_barbell_and_surface}
This section studies the relations of Budney--Gabai's barbell diffeomorphisms on $S^1\times D^3$ and the map $\Psi$ defined in Section \ref{sec_embedding_space_homotopy}.

From now on, we denote $Y_0 = S^1\times D^2$. 
Budney--Gabai introduced a homomorphism from $\pi_0\Homeo_\partial (I\times Y_0)$ to $\pi_2\Emb_\partial(I,I\times Y_0)$ in \cite{budney2019knotted} (cf. \cite[Section 2]{LXZ2025}), which we will denote by $\mathcal{S}$. We briefly recall the definition of $\mathcal{S}$ here. 
Pick a point $(t_0,x_0)\in S^1\times D^2$ and let $i_0:D^2\to Y_0$ be given by $i_0(x)=(t_0,x)$. 
We have the composition $\mathscr{I}\circ i_0:D^2\to \Emb_\dagger(I, I\times Y_0)$ (recall that $\mathscr{I}$ was defined in \eqref{eqn_def_mathscrI}). 
Given $\varphi\in \Homeo_\partial (I\times Y_0)$, consider the map
$$
\varphi_\ast \circ \mathscr{I}\circ i_0:D^2\to \Emb_\dagger(I, I\times Y_0)
$$
where $\varphi_\ast :\Emb_\dagger(I, I\times Y_0)\to \Emb_\dagger(I, I\times Y_0)$ is induced by $\varphi$. Since
$\varphi$ is supported in the interior of $I\times Y_0$, $\varphi_\ast \circ \mathscr{I}\circ i_0$
and  $\mathscr{I}\circ i_0$ coincide near the boundary of $D^2$. Therefore, we may glue the two maps
to obtain $\mathcal{S}(\varphi)\in\pi_2(\Emb_\dagger(I, I\times Y_0), \mathscr{I}(t_0,x_0))$. It is clear that $\mathcal{S}(\varphi)$ only depends on the isotopy class $[\varphi]\in\pi_0\Homeo_\partial(I\times Y_0)$, so we also denote it by $\mathcal{S}([\varphi])$.

Sometimes, it is important to specify the base point $t_0\in S^1$ in the above definition. In this case, we will denote the map $\mathcal{S}$ by $\mathcal{S}_{t_0}$. 


Now suppose $\sigma:(\Sigma,\partial\Sigma)\looparrowright (Y,\partial Y)$ is a properly immersed, connected, oriented surface. If $Y$ is closed, then this assumption implies $\partial \Sigma=\emptyset$. Suppose $\gamma:S^1\hookrightarrow Y$ is a smoothly embedded oriented circle such that the induced map 
\[
\gamma_*:H_1(S^1;\bZ) \to H_1(Y;\bZ)/\tor
\] 
is injective. Let $\nu(\gamma):Y_0=S^1\times D^2\to Y$ be an orientation-preserving embedding whose core circle is $\gamma$. 
Let $\varphi\in \Homeo_\partial(I\times Y_0)$, let $\nu(\gamma)_*(\varphi)\in \Diff(I\times Y)$ be the extension of $\varphi$ to $I\times Y$ via 
$\id_I\times \nu(\gamma):I\times Y_0\to I\times Y.$
Let $\varphi_*:\Emb_\dagger(I,I\times Y)\to \Emb_\dagger(I,I\times Y)$ denote the homeomorphism induced by $\nu(\gamma)_*\varphi$. 
Let $f=\mathscr{I}\circ \sigma:\Sigma\to \Emb_\dagger(I,I\times Y)$. 
The next proposition compares the homotopy classes of $f$ and $\varphi_*\circ f$.

\begin{Proposition}
\label{prop_interior_connected_sum_formula_at_each_intersection}
Suppose $\gamma(S^1)$ intersects $\sigma(\Sigma)$ transversely at $k$ distinct points $p_1,\dots,p_k\in Y$, with signs $\epsilon_1,\dots,\epsilon_k$. Suppose that $\sigma$ is an embedding on a neighborhood of 
\[
\sigma^{-1}(\{p_1,\dots,p_k\})\subset \Sigma.
\]
Then the map $\varphi_*\circ f:\Sigma\to \Emb_\dagger(I,I\times Y)$ is homotopic relative to $\partial \Sigma$ to the result of interior connected sums of $f$ with $\alpha_i \in \pi_2 (\Emb_\dagger(I,I\times Y),\mathscr{I}(p_i))$ 
at $\sigma^{-1}(p_i)$ ($i=1,\cdots,k)$ respectively, where each $\alpha_i$ satisfies
\begin{equation}
\label{eqn_alphi_i_Psi_in_connected_sum}
\Psi(\alpha_i) = \epsilon_i\cdot  \nu(\gamma)_*\circ \Psi\circ \mathcal{S}_{\gamma^{-1}(p_i)}([\varphi]).
\end{equation}
\end{Proposition}

Here, the map $\Psi$ on the left-hand side of \eqref{eqn_alphi_i_Psi_in_connected_sum} is defined on $\pi_2(\Emb_\dagger(I,I\times Y),\mathscr{I}(p_i))$, and the map $\Psi$ on the right-hand side is defined on $\pi_2(\Emb_\dagger(I,I\times Y_0), \mathscr{I}(\gamma^{-1}(p_i),x_0))$, where $x_0\in D^2$ is a base point on the $D^2$ factor of $Y_0$. The map $\mathcal{S}_{\gamma^{-1}(p_i)}$ is the map $\mathcal{S}$ defined above with respect to the base point $(\gamma^{-1}(p_i),x_0)\in S^1\times D^2=Y_0$. The map $\nu(\gamma)_*:\mathcal{W}(Y_0)\to \mathcal{W}(Y)$ is the map on the codomains of $\Psi$ induced by $\nu(\gamma)$, as defined in Section \ref{subsec_property_Psi}. 

\begin{proof}
By the construction of $\mathcal{S}$ and  $\varphi_*\circ f$, it is clear that
 $\varphi_*\circ f$ is homotopic relative to $\partial \Sigma$ to the interior connected sum of $f$ with $\alpha_i $ at $\sigma^{-1}(p_i)$ ($i=1,\cdots,k)$ respectively, where
 $$
 \alpha_i = \epsilon_i \cdot \nu(\gamma)_* \circ \mathcal{S}_{\gamma^{-1}(p_i)}([\varphi]) \in \pi_2 (\Emb_\dagger(I,I\times Y),\mathscr{I}(p_i)),
 $$
 and $\nu(\gamma)_*: \pi_2 (\Emb_\dagger(I,I\times Y_0),\mathscr{I}(\gamma^{-1}(p_i),x_0))\to \pi_2 (\Emb_\dagger(I,I\times Y),\mathscr{I}(p_i))$ denotes the map on the homotopy groups of embedding spaces induced by $\nu(\gamma)$. 
 Now the desired proposition follows from Proposition \ref{prop_naturality_Psi}.
\end{proof}

Note that the codomains of $\Psi$ and the map $\nu(\gamma)_*:\mathcal{W}(Y_0)\to\mathcal{W}(Y)$ are given by the homology groups of $Y_0$ and $Y$, and hence do not depend on the choice of the base point. We also have the following observation.

\begin{Lemma}
\label{lem_Stvarphi_independent_t}
    The value of $\Psi\circ \mathcal{S}_{t}([\varphi])$ does not depend on the base point $t\in S^1$. 
\end{Lemma}
\begin{proof}
    Let $t_1,t_2\in S^1$, and let $r:Y_0\to Y_0$ be the rotation in the $S^1$ factor such that $r(\{t_1\}\times D^2)=\{t_2\}\times D^2$. Then by Proposition \ref{prop_naturality_Psi}, we have $\Psi\circ \mathcal{S}_{t_2}([\varphi])=\Psi\circ \mathcal{S}_{t_1}([r^{-1}\circ \varphi\circ r])$. Since $r^{-1}\circ \varphi\circ r$ is isotopic to $\varphi$ relative to boundary, we have $\Psi\circ \mathcal{S}_{t_1}([\varphi]) = \Psi\circ \mathcal{S}_{t_2}([\varphi])$. 
\end{proof}

By Lemma \ref{lem_Stvarphi_independent_t}, it makes sense to refer to $\Psi\circ \mathcal{S}([\varphi])\in \mathcal{W}(Y_0)$ without specifying the base point $t\in S^1$. 
As a result, we have the following corollary of Proposition \ref{prop_interior_connected_sum_formula_at_each_intersection}.     
Let $\gamma$, $\sigma$, $f$, $\varphi_*$ be given by the setup above Proposition \ref{prop_interior_connected_sum_formula_at_each_intersection}. Here, we do not need to assume that $\gamma$ and $\sigma$ have transverse intersections.
\begin{Corollary}
\label{cor_interior_connected_sum_formula_on_surface_homotopy}
Let $q\in\inte(\Sigma)$ be a fixed point. Then
     $\varphi_*\circ f:\Sigma\to \Emb_\dagger(I,I\times Y)$ is homotopic relative to $\partial \Sigma$ to the interior connected sum of $f$ at $q\in\inte(\Sigma)$ with $\alpha \in \pi_2 (\Emb_\dagger(I,I\times Y),\mathscr{I}\circ \sigma(q))$, where $\alpha$ satisfies
    \[
    \Psi(\alpha) = \# (\sigma\cap \gamma) \cdot  \nu(\gamma)_*\circ \Psi\circ \mathcal{S}([\varphi]).
    \]
    Here, $\# (\sigma\cap \gamma)$ denotes the algebraic intersection number of $\sigma$ and $\gamma$.
\end{Corollary}

\begin{proof}
After isotopy of $\gamma$, we may assume $\gamma(S^1)$ intersects $\sigma(\Sigma)$ transversely at $k$ points
$p_1,\cdots,p_k$, and that $\sigma$ is injective on a neighborhood of $\sigma^{-1}(\{p_1,\dots,p_k\})\subset \Sigma$.

  According to Proposition \ref{prop_interior_connected_sum_formula_at_each_intersection},
  $\varphi_*\circ f$ is homotopic to the result of interior connected sums of $f$ with $\alpha_i \in \pi_2 (\Emb_\dagger(I,I\times Y),\mathscr{I}(p_i))$ at $\sigma^{-1}(p_i)$ ($i=1,\cdots,k)$ respectively, where each $\alpha_i$ satisfies \eqref{eqn_alphi_i_Psi_in_connected_sum}. Pick an arbitrary path $\eta_i$ from $\sigma^{-1}(p_i)$ to $q$ in $\Sigma$ for each $i$. 
  Let $\alpha_i^{\eta_i}\in \pi_2 (\Emb_\dagger(I,I\times Y),\mathscr{I}\circ \sigma(q))$ denote the image of $\alpha_i$ under the base-point-change isomorphism from $\pi_2 (\Emb_\dagger(I,I\times Y),\mathscr{I}(p_i))$ to $\pi_2 (\Emb_\dagger(I,I\times Y),\mathscr{I}\circ \sigma(q))$ induced by the path $\mathscr{I}\circ\sigma\circ \eta_i$.
  Then
  $\varphi_*\circ f$ is homotopic to the interior connected sum of $f$ with $\alpha = \sum_i \alpha_i^{\eta_i}$ at $q$. By Proposition \ref{prop_compare_Psi_all_base_point_Y}, we have $\Psi(\alpha_i^{\eta_i}) = \Psi(\alpha_i)$. So 
  \[
    \Psi(\alpha) = \sum_i \Psi(\alpha_i) = \# (\sigma\cap \gamma) \cdot  \nu(\gamma)_*\circ \Psi\circ \mathcal{S}([\varphi]).
  \]
  Hence, the corollary is proved.
\end{proof}

Corollary \ref{cor_interior_connected_sum_formula_on_surface_homotopy} extends almost verbatim if we replace $\gamma$ with a disjoint union of embedded circles. Suppose $\gamma_1,\dots,\gamma_k:S^1\to Y$ are disjoint smoothly embedded oriented circles in $Y$, such that each $\gamma_i$ induces an injection $H_1(S^1)\to H_1(Y;\mathbb{Z})/\tor$, and let $\nu(\gamma_i): Y_0=S^1\times D^2 \to Y$ be an orientation-preserving embedding with the core circle given by $\gamma_i$. Assume that the $\nu(\gamma_i)$'s have disjoint images. Let $\varphi_i\in \Homeo_\partial(I\times Y_0)$, let $\nu(\gamma_i)_*(\varphi)\in \Diff(I\times Y)$ be the extension of $\varphi_i$ to $I\times Y$ via $\id_I\times \nu(\gamma_i)$.  Note that $\nu(\gamma_i)_*(\varphi_i)$ have disjoint supports, so they are commutative.  Let $ \varphi$ be the product of all $\nu(\gamma_i)_*(\varphi_i)$. Let 
\[
\varphi_*:\Emb_\dagger(I,I\times Y)\to \Emb_\dagger(I,I\times Y)
\]
denote the homeomorphism induced by $\varphi$.
Let $\sigma$ be a proper immersion of $(\Sigma,\partial \Sigma)$ in $(Y,\partial Y)$, and let $f=\mathscr{I}\circ \sigma:\Sigma\to \Emb_\dagger(I,I\times Y)$. Then we have the following result.

\begin{Proposition}
\label{prop_interior_connected_sum_formula_on_surface_homotopy_multi_curves}
Let $\varphi_i,\gamma_i,\varphi,f$ be as above. Let $q\in \inte(\Sigma)$ be a fixed point.
    Then the map $\varphi_*\circ f$ is homotopic relative to $\partial \Sigma$ to the interior connected sum of $f$ at $q$ with $\alpha \in \pi_2 (\Emb_\dagger(I,I\times Y),\mathscr{I}\circ \sigma(q))$, where $\alpha$ satisfies
    \[
    \Psi(\alpha) = \sum_i \# (\sigma\cap \gamma_i) \cdot  \nu(\gamma_i)_*\circ \Psi\circ \mathcal{S}([\varphi_i]).
    \]
    Here, $\# (\sigma\cap \gamma_i)$ denotes the algebraic intersection number of $\sigma$ and $\gamma_i$.
\end{Proposition}

\begin{proof}
Let $q_1,\dots,q_k$ be distinct points in $\inte(\Sigma)$.
    By the same argument as Corollary \ref{cor_interior_connected_sum_formula_on_surface_homotopy}, each $\nu(\gamma_i)_*(\varphi_i)$ changes the homotopy class of $f$ by taking an interior connected sum with some $\alpha_i \in \pi_2 (\Emb_\dagger(I,I\times Y),\mathscr{I}\circ \sigma(q_i))$ at $q_i$, where  $\Psi(\alpha_i) =  \# (\sigma\cap \gamma_i) \cdot  \nu(\gamma_i)_*\circ \Psi\circ \mathcal{S}([\varphi_i])$. Hence, the desired result follows by moving the attaching point of each interior connected sum from $q_i$ to $q$ via an arc on $\Sigma$ and invoking Proposition \ref{prop_compare_Psi_all_base_point_Y}.
\end{proof}

We also record the following results by Budney--Gabai for later reference. Recall that we denote $Y_0=S^1\times D^2$.

\begin{Theorem}[\cite{budney2019knotted}]
\label{thm_BG_construction_S1xD3}
    There exists a sequence $\varphi_i\in \Diff_{PI}(I\times Y_0)$, such that 
\begin{enumerate}
    \item $\Psi\circ\mathcal{S}([\varphi_i])\in \mathcal{W}(Y_0)$ are linearly independent over $\mathbb{Z}$.
    \item Suppose $\iota:Y_0\to Y_0$ is an orientation-preserving homeomorphism that induces $(-\id)$ on $H_1(Y_0;\mathbb{Z})\cong \bZ$. Then 
    \[
            \iota_*\circ \Psi\circ\mathcal{S}([\varphi_i])=-\Psi\circ\mathcal{S}([\varphi_i])
    \]
    for all $i$. 
\end{enumerate}
\end{Theorem}

\begin{proof}
    When applied to the special case of $Y=Y_0$, the composition map $\Psi\circ \mathcal{S}$ is the same as the ``quotient of the $W_3$ invariant'' on $\pi_0\Homeo_\partial(S^1\times D^3)$ constructed by Budney--Gabai in \cite[Section 4]{budney2025automorphism}. The ``quotient of the $W_3$ invariant'' is denoted by $W_3'$ in \cite{budney2025automorphism}.  In \cite{budney2019knotted}, a sequence of elements $\varphi_i\in \Diff_{PI}(S^1\times D^3)$ is constructed, and it was shown in \cite[Theorem 8.3]{budney2019knotted} that the $W_3$ invariants of $\varphi_i$ are linearly independent.  It was noted in \cite[Section 4]{budney2025automorphism} that after the quotient on the codomain, the $W_3'$ invariants are also linearly independent. This yields Statement (1) of the theorem. 
    
    The above properties were proved by directly computing the $W_3$ and $W_3'$ invariants. The result of the computations can be found in \cite[Corollary 8.2]{budney2019knotted}. Statement (2) of the theorem is an immediate consequence of the result of the computation.
\end{proof}

Theorem \ref{thm_BG_construction_S1xD3} has the following consequence.

\begin{Corollary}
\label{cor_homeo_reverse_H1_on_Y0}
    Let $\varphi_i$ be given by Theorem \ref{thm_BG_construction_S1xD3}. Let $\iota:Y_0\to Y_0$ be an orientation-preserving diffeomorphism that induces $(-\id)$ on $H_1(Y_0;\mathbb{Z})$, let $\hat\iota = \id_I\times\iota: I\times Y_0\to I\times Y_0$. Then 
    \[
\Psi\circ\mathcal{S}([\varphi_i])=\Psi\circ\mathcal{S}([\hat\iota\circ \varphi_i\circ\hat\iota^{-1}])
    \]
\end{Corollary}

\begin{proof}
Let $(t_0,x_0)\in S^1\times D^2=Y_0$ be the base point.
    After isotopy, assume that $\iota(\{t_0\}\times D^2) = (\{t_0\}\times D^2)$ and that $\iota$ fixes $(t_0,x_0)$. Then $\iota$ restricts to an orientation-reversing diffeomorphism on $\{t_0\}\times D^2$. 
    
    Let $\hat{\iota}_*:\pi_2\Emb_\dagger(I,I\times Y_0)\to \pi_2\Emb_\dagger(I,I\times Y_0)$ be the induced map by $\hat\iota$. Then $\mathcal{S}([\hat\iota\circ \varphi_i\circ\hat\iota^{-1}]) = -\hat{\iota}_*\circ  \mathcal{S}([\varphi])$. The negative sign comes from the fact that $\iota$ restricts to an orientation-reversing diffeomorphism on $\{t_0\}\times D^2$, so a coordinate change by $\hat\iota$ reverses the sign in the definition of $\mathcal{S}$. Therefore, by Theorem \ref{thm_BG_construction_S1xD3} (2) and Proposition \ref{prop_naturality_Psi}, we have 
    \[
\Psi\circ \mathcal{S}([\hat\iota\circ \varphi_i\circ\hat\iota^{-1}]) = -\Psi\circ \hat{\iota}_*\circ  \mathcal{S}([\varphi]) = -\hat{\iota}_*\circ \Psi\circ  \mathcal{S}([\varphi]) = \Psi\circ  \mathcal{S}([\varphi]) .
    \]
\end{proof}

\section{The case when $\partial Y\neq \emptyset$}
\label{sec_boundary_non_empty}

This section proves Theorem \ref{thm_main_irred_case} in the case when $\partial Y\neq \emptyset$. In this case, the theorem can be restated as follows. Recall that 
$\pi_0 \Diff_{PI}(I\times Y)$, $\pi_0 \Diff_{\partial}(I\times Y)$, $\pi_0\Homeo_{\partial}(I\times Y)$
are all abelian. 
\begin{Theorem}
\label{thm_main_non-closed_case}
    Suppose $Y$ is a compact, connected, oriented, irreducible $3$--manifold such that $\partial Y\neq \emptyset$ and $\pi_1(Y)$ is infinite. Then there exists a smooth embedding $i:S^1\times D^2=Y_0\to Y$, such that the image of the composition map 
    \[
        \pi_0 \Diff_{PI}(I\times Y_0)\xrightarrow{(\id_I\times i)_*} \pi_0 \Diff_{PI}(I\times Y) \to \pi_0\Homeo (I\times Y,I\times \partial Y)
    \]
is of infinite rank, where $(\id_I\times i)_*$ denotes the extension to $I\times Y$ via $\id_I\times i$.
\end{Theorem}

\begin{proof}
    We first show that the assumptions imply $b_1(Y)>0$. Note that since $\pi_1(Y)$ is infinite, we know that $Y\not\cong D^3$. Since $Y$ is irreducible, $\partial Y$ does not contain any connected component that is diffeomorphic to $S^2$. Since $\partial Y\neq \emptyset$, we have $\rank H_1(\partial Y)>0$. By a standard result in $3$--manifold topology (see, for example, \cite[Lemma 3.5]{hatcher2007notes}), the rank of the image of $H_1(\partial Y)\to H_1(Y)$ is $\frac12 \rank H_1(\partial Y)$. So $b_1(Y)>0$.

    Let $\sigma:(\Sigma,\partial \Sigma)\to (Y,\partial Y)$ be a properly immersed, connected, oriented surface and let $\gamma:S^1\to \inte(Y)$ be a smooth embedding such that the algebraic intersection number of $\Sigma$ and $\gamma$ is non-zero. Since $b_1(Y)>0$, such a pair $(\sigma,\gamma)$ exists. After tubing $\Sigma$ to $\partial Y$, we may assume that $\partial \Sigma\neq \emptyset$. 

    Let $\nu(\gamma):S^1\times D^2 = Y_0\to Y$ be an orientation-preserving embedding with core circle given by $\gamma$.
    Let $\varphi_1,\varphi_2,\dots\in \Diff_{PI}(I\times Y_0)$ be given by Theorem \ref{thm_BG_construction_S1xD3}. 
    Let $\nu(\gamma)_*(\varphi_i)$ be the extension of $\varphi_i$ to $I\times Y$ via $\id_I\times \nu(\gamma)$. 
    
    We claim that $\nu(\gamma)_*(\varphi_i)$'s are linearly independent over $\mathbb{Z}$ in $\pi_0\Homeo(I\times Y,I\times \partial Y)$. Suppose $0\neq (c_k) \in \oplus_\infty \mathbb{Z}$ and let $\varphi=\prod_k \varphi_k^{c_k}$. The goal is to show that $\nu(\gamma)_*(\varphi)$ is non-trivial in $\pi_0\Homeo(I\times Y,I\times \partial Y)$. 

    Consider the map $f = \mathscr{I}\circ \sigma: \Sigma \to \Emb_\dagger(I,I\times Y)$. Let $\varphi_*: \Emb_\dagger(I,I\times Y) \to \Emb_\dagger(I,I\times Y)$ be the self-homeomorphism of $\Emb_\dagger(I,I\times Y)$ induced by $\nu(\gamma)_*\varphi:I\times Y \to I\times Y$. We first show that $f$ and $\varphi_*\circ f$ are not homotopic relative to $\partial \Sigma$. 
    
    By Corollary \ref{cor_interior_connected_sum_formula_on_surface_homotopy}, there exists $\alpha\in \pi_2\Emb_\dagger(I,I\times Y)$ such that $\varphi_* f$ is homotopic to an interior connected sum of $f$ with $\alpha$, where $\alpha$ satisfies
    \[
    \Psi(\alpha) = \#(\sigma\cap\gamma)\cdot \nu(\gamma)_*\circ \Psi\circ \mathcal{S}([\varphi]).
    \]
    By Theorem \ref{thm_BG_construction_S1xD3} and the definition of $\varphi$, we know that $\Psi\circ \mathcal{S}([\varphi])\neq 0\in \mathcal{W}(Y_0)$. By Proposition \ref{prop_invariance_Psi_under_pi1_Y}, the map $\nu(\gamma)_*:\mathcal{W}(Y_0)\to \mathcal{W}(Y)$ is injective. We also have $\#(\sigma\cap\gamma)\neq 0$ by the definitions of $\sigma$ and $\gamma$. So $\Psi(\alpha)\neq 0$.

By Proposition \ref{prop_invariance_Psi_under_pi1_Y}, all elements of the form 
$\beta - \beta^u$, 
where 
\[
\beta\in \pi_2\Emb_\dagger(I,I\times Y),\quad  u\in \Ima(i_*:\pi_1(\Sigma)\to \pi_1\Emb_\dagger(I,I\times Y)),
\]
are in the kernel of $\Psi$. 
So $\alpha$ is not contained in the subgroup generated by the elements of the form $\beta - \beta^u$. 
By Lemma \ref{lem_surface_hmtpy_rel_boundary}, we conclude that $f$ and $f'$ are not homotopic to each other relative to $\partial \Sigma$. 

Note that $\nu(\gamma)_*(\varphi)$ is homotopic to $\id$ relative to $\{0,1\}\times Y$. 
Therefore, by Lemma \ref{lem_Embdagger_obstruct_isotopy}, $\nu(\gamma)_*(\varphi)$ is non-trivial in $\pi_0\Homeo(I\times Y,I\times \partial Y)$. This proves the desired result. 
\end{proof}

\begin{remark}
        The proof for Theorem \ref{thm_main_non-closed_case} relies on the assumption that $\partial Y\neq \emptyset$ in two steps: First, the proof invokes Lemma \ref{lem_surface_hmtpy_rel_boundary}, which only applies to surfaces with non-empty boundaries. Second, the assumptions of Theorem \ref{thm_main_non-closed_case} imply $b_1(Y)>0$, which is not true in general in the closed case.  Section \ref{sec_closed_irreducible} will address these issues and generalize the argument to the case when $\partial Y=\emptyset$. 
\end{remark}

\section{The case when $\partial Y= \emptyset$}
\label{sec_closed_irreducible}

To prove Theorem \ref{thm_main_irred_case} when $\partial Y=\emptyset$, we invoke the following theorem from the literature.

\begin{Theorem}[Agol \cite{agol2013virtual}, Wise \cite{wise2009research}, Luecke \cite{luecke1988finite}, et al]
\label{thm_virtual_positive_b1}
    Suppose $Y$ is a closed, oriented, irreducible $3$--manifold such that $\pi_1(Y)$ is infinite. Then there exists a finite cover $\widetilde{Y}$ of $Y$ such that $b_1(\widetilde{Y})>0$.
\end{Theorem}

\begin{proof}
    If $Y$ is hyperbolic, then the virtually fibered theorem \cite{agol2013virtual,wise2009research} states that $Y$ admits a finite cover that is a surface bundle over $S^1$, which implies the desired result. If $Y$ admits a non-trivial JSJ decomposition, the desired result is given by  \cite{luecke1988finite}. If $Y$ is Seifert fibered, then by the assumption that $Y$ is irreducible and has an infinite $\pi_1$, we know that the orbit space of $Y$ has a non-positive orbifold Euler number, so $Y$ can be covered by an $S^1$ bundle over an oriented surface with a positive genus, and the desired result follows.
\end{proof}

\begin{remark}
The virtually fibered property is also known to hold for mixed $3$--manifolds (i.e., manifolds that are neither hyperbolic nor graph) \cite{przytycki2018mixed} and for graph manifolds that admit Riemannian metrics with non-positive sectional curvature \cite{liu2013virtual}.
\end{remark}

The next result shows that we can take the covering space $\widetilde{Y}$ in Theorem \ref{thm_virtual_positive_b1} to be a normal covering. 

\begin{Corollary}
    \label{cor_normal_cover_positive_b1}
    Suppose $Y$ is a closed, oriented, irreducible $3$--manifold such that $\pi_1(Y)$ is infinite. Then there exists a \emph{normal} finite cover $\widetilde{Y}$ of $Y$ such that $b_1(\widetilde{Y})>0$.
\end{Corollary}

\begin{proof}
    Let $\widetilde{Y}'$ be a finite cover of $Y$ with $b_1(\widetilde{Y}')>0$, as given by Theorem \ref{thm_virtual_positive_b1}. Fix a base point of $\widetilde{Y}'$ and let its image in $Y$ be the base point of $Y$.  Let $H\subset \pi_1(Y)$ be the image of $\pi_1(\widetilde{Y}')$ in $\pi_1(Y)$. Since $H$ is a subgroup of $\pi_1(\widetilde{Y}')$ with finite index, there are at most finitely many subgroups of $\pi_1(\widetilde{Y}')$ with the form $gHg^{-1}$, $g\in \pi_1(Y)$. Let $K = \cap_{g\in \pi_1(Y)} g H g^{-1}$, then $K$ is a normal subgroup of $\pi_1(Y)$ with finite index. Let $\widetilde{Y}$ be the cover of $Y$ determined by $K$, then $\widetilde{Y}$ is a normal cover. There is a finite covering map $p:\widetilde{Y}\to \widetilde{Y'}$. The map $p$ induces an injection $H^1(\widetilde{Y'};\mathbb{R}) \to H^1(\widetilde{Y};\mathbb{R})$, because pulling-back by $p$ defines an injection on the space of harmonic forms. Therefore, $b_1(\widetilde{Y})\ge b_1(\widetilde{Y}')>0$. So the corollary is proved. 
\end{proof}

Note that if $\widetilde{Y}$ is a finite cover of $Y$, then every element of $\Diff_{PI}(I\times Y)$ admits a unique lifting to $\Diff_{PI}(I\times \widetilde{Y})$ by the homotopy lifting property of covering spaces.
Recall that $\Homeo_0(I\times Y)$ denotes the group of homeomorphisms of $I\times Y$ that are homotopic to $\id$, without necessarily fixing the boundary. 
The next lemma allows us to define a pull-back homomorphism from $\pi_0\Homeo_{0}(I\times Y)$ to $\pi_0\Homeo_{0}(I\times \widetilde{Y})$. 

\begin{Lemma}
\label{lem_pull_back_diff_PI}
    Suppose $Y$ is a closed oriented $3$--manifold, and $p:\widetilde{Y}\to Y$ is a normal covering map. Also assume that every deck transformation on $\widetilde{Y}$ that is homotopic to the identity is isotopic to the identity. Let $\varphi\in \Homeo_{0}(I\times Y)$. Then there exists a lifting $\tilde\varphi\in \Homeo_0(I\times \widetilde{Y})$ such that the diagram commutes
    \[
\begin{tikzcd}
        I\times \widetilde{Y} \arrow[r,"\tilde{\varphi}"] \arrow[d, "\id_I\times p"] & I\times \widetilde{Y}\arrow[d, "\id_I\times p"] \\
        I\times Y \arrow[r,"\varphi"] & I\times Y.
\end{tikzcd}
    \]
    Moreover, the isotopy class of $\tilde\varphi$ is determined by $\varphi$.
\end{Lemma}

\begin{proof}
    The existence of $\tilde{\varphi}$ is given by the lifting of the homotopy to the identity to the covering space. If $\tilde{\varphi}_1$ and $\tilde{\varphi}_2$ are two such liftings, then they must differ by $\id_I$ times a deck transformation that is homotopic to the identity. By the assumptions on $\widetilde{Y}$, the diffeomorphism must be isotopic to the identity. Therefore, the isotopy class of $\tilde\varphi$ is determined by $\varphi$.
\end{proof}
\begin{remark}
    By a theorem of Waldhausen \cite{waldhausen1968irreducible}, if $\widetilde{Y}$ is Haken, then every homeomorphism of $\widetilde{Y}$ that is homotopic to the identity is isotopic to the identity.
\end{remark}

By Lemma \ref{lem_pull_back_diff_PI}, if $p:\widetilde{Y}\to Y$ is a normal covering map and $\widetilde{Y}$ is Haken, we have a well-defined homomorphism $p^*:\pi_0\Homeo_{PI}(I\times Y) \to \pi_0\Homeo_{PI}(I\times \widetilde{Y})$ given by $p^*([\varphi]) = [\tilde{\varphi}]$. 

Let $Y$ be closed, oriented, irreducible, with an infinite $\pi_1$, and let $\widetilde{Y}$ be a normal cover of $Y$ such that $b_1(\widetilde{Y})>0$. Note that since $Y$ is irreducible and $\pi_1(Y)$ is infinite, $Y$ is aspherical. Since $\pi_i(\widetilde{Y})\cong \pi_i(Y)$ for all $i\ge 2$, we know that $\widetilde{Y}$ is also aspherical, so it is irreducible. Since $b_1(\widetilde{Y})>0$, we know that $\widetilde{Y}$ is Haken. So the pull-back map $\pi_0\Homeo_{PI}(I\times Y)\to \pi_0\Homeo_{PI}(I\times \widetilde{Y})$ is well-defined in this case. 

\subsection{The case when $\widetilde{Y}$ is not Seifert fibered}
In this subsection, we prove Theorem \ref{thm_main_irred_case} when $Y$ is closed and $\widetilde{Y}$ is not Seifert fibered.
We will apply Lemmas \ref{lem_closed_surface_basepoint_connected_sum_homotopy} and \ref{lem_surface_hmtpy_closed} to obstruct the homotopy of surfaces in $\Emb_\dagger(I,I\times Y)$. In order to verify the assumptions of Lemma \ref{lem_surface_hmtpy_closed}, we need to study the centralizers of subgroups of $\pi_1\Emb_\dagger(I,I\times Y)$.

\begin{Proposition}
\label{prop_trivial_centralizer_in_Emb_dagger}
    Suppose $Y$ is a closed, oriented, irreducible $3$--manifold, and $\pi_1(Y)$ contains two elements $\xi_1$, $\xi_2$, such that 
    \begin{enumerate}
        \item $\xi_1$ is of infinite order, and for each positive integer $m$, the centralizer of $(\xi_1)^m$ in $\pi_1(Y)$ is the infinite cyclic group generated by $\xi_1$. 
        \item $\xi_2$ is not commutative with $\xi_1$, and $(\xi_2)^2 \neq 1$. 
    \end{enumerate} 
    Then the subgroup of $\pi_1\Emb_\dagger(I,I\times Y)$ generated by $\mathscr{I}_*(\xi_1),\mathscr{I}_*(\xi_2)$ has a trivial centralizer in $\pi_1\Emb_\dagger(I,I\times Y)$. 
\end{Proposition}

\begin{remark}
To be more precise about the base points of the homotopy groups in the statement of Proposition \ref{prop_trivial_centralizer_in_Emb_dagger}, fix $q\in Y$ to be the base point of $Y$,  let $\mathscr{I}(q)$ be the base point of $\Emb_\dagger(I,I\times Y)$. Note that the triviality of the centralizer does not depend on the choice of $q$. We will omit the base points to simplify notation.
\end{remark}
\begin{proof}
    The assumptions imply that $\pi_1(Y)$ is infinite and $\pi_2(Y)=0$. By Proposition \ref{prop_decomposition_pi_1_emb_dagger}, we have 
    \[
    \pi_1\Emb_\dagger(I,I\times Y) \cong \bZ[\pi_1(Y)\setminus\{1\}]\rtimes \pi_1(Y),
    \]
    and $\mathscr{I}_*:\pi_1(Y)\to \pi_1\Emb_\dagger(I,I\times Y)$ is the inclusion to the component in the semidirect product. In particular, every element in $\pi_1\Emb_\dagger(I,I\times Y)$ can be uniquely represented by a product $xy$, where $x\in \bZ[\pi_1(Y)\setminus\{1\}]$ and $y \in \pi_1(Y)$. For $z\in \pi_1(Y)$, we have 
    \[
    z(xy)z^{-1} = (zxz^{-1})(zyz^{-1}),
    \]
    so $z$ is commutative with $xy$ if and only if $z$ is commutative with both $x$ and $y$. 

    By the assumptions, the subgroup of $\pi_1(Y)$ generated by $\xi_1,\xi_2$ has a trivial centralizer in $\pi_1(Y)$. So we only need to show that if $x\in \bZ[\pi_1(Y)\setminus\{1\}]$ is invariant under the conjugation actions of both $\xi_1$ and $\xi_2$, then $x=0$.

Suppose $x=\sum_{i=1}^n a_i \eta_i\in \mathbb{Z}[\pi_1(Y)\setminus \{1\}]$, where $a_i\in\mathbb{Z}, \eta_i\in \pi_1(Y)\setminus\{1\}$, is invariant under the conjugation actions of $\xi_1,\xi_2$, we show that $x = 0$. Without loss of generality, assume $\eta_1,\dots,\eta_n$ are mutually distinct and $a_i\neq 0$ for all $i$. The goal is to deduce a contradiction when $n>0$. Since $\sum a_i \eta_i$ is invariant under the conjugation of $\xi_1$, the conjugation action of $\xi_1$ defines a permutation on the set $\{\eta_i\}_{1\le i \le n}$. Therefore, $(\xi_1)^{n!}$ is commutative with each $\eta_i$. By the assumptions, there exist non-zero integers $b_1,\dots,b_n$ such that $\eta_i=(\xi_1)^{b_i}$. 

Since $\sum a_i \eta_i$ is also invariant under the conjugation action of $\xi_2$, there exists a sequence of non-zero integers $c_1,\dots,c_n$ such that
\[
\xi_2 \cdot (\xi_1)^{b_i}\cdot (\xi_2)^{-1} = (\xi_1)^{c_i}
\]
and that $(c_1,\dots,c_n)$ equals a permutation of $(b_1,\dots,b_n)$. The above equation implies 
\[
(\xi_2 \cdot \xi_1 \cdot (\xi_2)^{-1})^{b_i} = (\xi_1)^{c_i},
\]
so $\xi_2 \cdot \xi_1 \cdot (\xi_2)^{-1}$ is commutative with $(\xi_1)^{c_i}$. By the assumption on $\xi_1$, there exists an integer $k$ such that 
\[
\xi_2 \cdot \xi_1 \cdot (\xi_2)^{-1} = (\xi_1)^k.
\]
Since $\xi_1$ is non-torsion, we have $c_i=k b_i$. Since $(c_1,\dots,c_n)$ is a permutation of $(b_1,\dots,b_n)$, we have $k=1$ or $-1$. If $k=1$, then $\xi_1$ is commutative with $\xi_2$, contradicting the assumption on $\xi_2$. If $k=-1$, we have 
\[
(\xi_2)^2\,\xi_1\,(\xi_2)^{-2} = \xi_1,
\]
so $(\xi_2)^2$ is commutative with $\xi_1$. By the assumptions, there exists a non-zero integer $l$ such that $(\xi_2)^2 = (\xi_1)^l$. So $\xi_2$ is commutative with $(\xi_1)^l$, therefore $\xi_2$ is an integer power of $\xi_1$, thus $\xi_2$ is commutative with $\xi_1$, which yields a contradiction.
\end{proof}

Now we can prove Theorem \ref{thm_main_irred_case} when $\widetilde{Y}$ satisfies the assumptions of Proposition \ref{prop_trivial_centralizer_in_Emb_dagger}. Note that since we assume $Y$ is closed, $\Homeo(I\times Y,I\times \partial Y)$ is the same as $\Homeo(I\times Y)$. 

\begin{Theorem}
\label{thm_main_irred_case_when_xi1_xi2_exist_in_pi1}
Suppose $Y$ is a closed, oriented, irreducible $3$--manifold, and $\pi_1(Y)$ is infinite. Let $\widetilde{Y}$ be a finite normal cover of $Y$ such that $b_1(\widetilde{Y})>0$. Suppose there exist $\xi_1,\xi_2\in \pi_1(\widetilde{Y})$ such that    
    \begin{enumerate}
        \item $\xi_1$ is of infinite order, and for each positive integer $m$, the centralizer of $(\xi_1)^m$ in $\pi_1(\widetilde{Y})$ is the infinite cyclic group generated by $\xi_1$. 
        \item $\xi_2$ is not commutative with $\xi_1$, and $(\xi_2)^2 \neq 1$. 
    \end{enumerate} 
Then there exists a smooth embedding $S^1\times D^2=Y_0\to Y$, such that the image of 
\[\pi_0\Diff_{PI}(I\times Y_0)\to \pi_0\Homeo(I\times Y)\] is of infinite rank. 
\end{Theorem}

\begin{proof}
    Let $\gamma:S^1\to \widetilde{Y}$ be an embedded simple closed curve, and let $\sigma:\Sigma\looparrowright \widetilde{Y}$ be a connected, closed, oriented, immersed surface whose algebraic intersection number with $\gamma$ equals $1$. By attaching handles to $\Sigma$, we may assume that the image of $\sigma_*:\pi_1(\Sigma)\to \pi_1(Y)$ contains $\xi_1$ and $\xi_2$. 
    
After perturbation, we may assume that the image of $\gamma$ in $Y$ is embedded, so $\gamma$ is disjoint from its images under the deck transformation group. Let $\gamma_1,\dots,\gamma_k$ be the images of $\gamma$ under deck transformations, where $\gamma_1=\gamma$. Let $\nu(\gamma_i):Y_0=S^1\times D^2 \to \widetilde{Y}$ be an orientation-preserving embedding whose core circle is $\gamma_i$. We may require that the maps $\nu(\gamma_i)$ are related to each other by deck transformations, and that the images of $\nu(\gamma_i)$ are disjoint from each other. Let $p:\widetilde{Y}\to Y$ be the covering map, then $p\circ \nu(\gamma_i)$ are the same for all $i$. Let $\hat{\nu}(\gamma) = p\circ \nu(\gamma_i)$. Then $\hat{\nu}(\gamma)$ is a smooth embedding of $S^1\times D^2$ in $Y$.

Let $\varphi_1,\varphi_2,\dots\in \Diff_{PI}(I\times Y_0)$ be given by Theorem \ref{thm_BG_construction_S1xD3}. Let $\tilde{\varphi}_{i,j}$ be the extension of $\varphi_j$ to $I\times \widetilde{Y}$ via $\id_{I}\times \nu(\gamma_i)$. Then, for a fixed value of $j$, the diffeomorphisms $\tilde{\varphi}_{i,j}$ for different values of $i$ are commutative to each other because they have disjoint supports. Let $\tilde{\varphi}_j = \prod_i \tilde{\varphi}_{i,j}$. Then 
the isotopy class of $\tilde{\varphi}_j$ is in the image of 

\begin{align}
        & \pi_0\Diff_{PI}(I\times Y_0)\xrightarrow{(\id_I\times \hat{\nu}(\gamma))_*}  \pi_0\Diff_{PI}(I\times Y)\nonumber\\
& \quad \to \pi_0\Homeo_{0}(I\times Y)\xrightarrow{(\id_I\times p)^*} \pi_0\Homeo_{0}(I\times \widetilde{Y}).
\label{eqn_composition_from_cylinder_to_covering_non_fibered_case}
\end{align}

We show that $\tilde{\varphi}_j$ are linearly independent over $\mathbb{Z}$ in $\pi_0\Homeo_{0}(I\times \widetilde{Y})$.
Let $0\neq (c_j)\in \oplus_\infty \mathbb{Z}$, let $\varphi = \prod_j \tilde{\varphi}_j^{c_j}\in \Homeo_{0}(I\times \widetilde{Y})$. We show that $\varphi$ is not isotopic to the identity in $\Homeo_{0}(I\times \widetilde{Y})$. 

Let $f=\mathscr{I}\circ \sigma:\Sigma\to \Emb_\dagger(I,I\times\widetilde{Y})$, let $\varphi_*: \Emb_\dagger(I,I\times\widetilde{Y})\to \Emb_\dagger(I,I\times\widetilde{Y})$ be the homeomorphism induced by $\varphi$. We obstruct the isotopy between $\varphi$ and $\id$ by showing that $f$ and $\varphi_*\circ f$ are not homotopic. 

By Proposition \ref{prop_interior_connected_sum_formula_on_surface_homotopy_multi_curves}, $\varphi_*\circ f$ is obtained from $f$ by an interior connected sum with $\alpha\in \pi_2\Emb_\dagger(I,I\times\widetilde{Y})$, where
\[
\Psi(\alpha) = \sum_i \#(\sigma\cap \gamma_i)\cdot \nu(\gamma_i)_*\circ \Psi\circ \mathcal{S}([\prod_j \varphi_j^{c_j}]).
\]
By Theorem \ref{thm_BG_construction_S1xD3} (1), we have $\Psi\circ \mathcal{S}([\prod_j \varphi_j^{c_j}])\neq 0 \in \mathcal{W}(Y_0)$. 

Let $r:\mathcal{W}(\widetilde{Y})\to \mathcal{W}(Y_0)$ be the retraction given by Proposition \ref{prop_embedding_and_codomain_of_Psi} with respect to the embedding $\nu(\gamma_1):Y_0\to \widetilde{Y}$. We claim that $r\circ \Psi(\alpha)$ is non-zero. 

To prove the claim, note that since the algebraic intersection number of $\gamma$ and $\sigma$ is $1$, the fundamental class of $\gamma$ is primitive in $H_1(\widetilde{Y};\mathbb{Z})/\tor$. In the following, we will use $H_1(\widetilde{Y})$ and $H_1(Y_0)$ to denote the homology groups in $\mathbb{Z}$ coefficients.  
Then for each $i$, the fundamental class of $\gamma_i$ is primitive in $\widetilde{Y}$. So, there are three possibilities:
\begin{enumerate}
    \item the image of 
\[
(\gamma_i)_*: \mathbb{Z}\cong H_1(Y_0)/\tor  \to H_1(\widetilde{Y})/\tor 
\]
intersects the image of $(\gamma_1)_*$ at $\{0\}$.
\item $(\gamma_i)_*=(\gamma_1)_*$ as maps from $H_1(Y_0)/\tor$ to $H_1(\widetilde{Y})/\tor$.
\item $(\gamma_i)_*=-(\gamma_1)_*$ as maps from  $H_1(Y_0)/\tor$ to $H_1(\widetilde{Y})/\tor$.
\end{enumerate}

For Case (1), the map $r$ equals $0$ on the image of $(\gamma_i)_*:\mathcal{W}(Y_0)\to \mathcal{W}(Y)$. So
\[
r(\#(\sigma\cap \gamma_i)\cdot \nu(\gamma_i)_*\circ \Psi\circ \mathcal{S}(\varphi)) = 0.
\]
For Case (2), we have $\#(\sigma\cap \gamma_i) = \#(\sigma\cap \gamma_1)$ and $\nu(\gamma_i)_*\circ \Psi\circ \mathcal{S}(\varphi)=\nu(\gamma_1)_*\circ \Psi\circ \mathcal{S}(\varphi)$, so 
\[
r(\#(\sigma\cap \gamma_i)\cdot \nu(\gamma_i)_*\circ \Psi\circ \mathcal{S}(\varphi)) = \#(\sigma\cap \gamma_1)\cdot \nu(\gamma_1)_*\circ \Psi\circ \mathcal{S}(\varphi).
\]
For Case (3), we have  $\#(\sigma\cap \gamma_i) = -\#(\sigma\cap \gamma_1)$ and $\nu(\gamma_i)_*=\nu(\gamma_1)_*\circ \tau_*$, where $\tau:Y_0\to Y_0$ is an orientation-preserving homeomorphism that induces $(-\id)$ on $H_1(Y_0)$. Hence by Theorem \ref{thm_BG_construction_S1xD3} (2), we still have
\[
r(\#(\sigma\cap \gamma_i)\cdot \nu(\gamma_i)_*\circ \Psi\circ \mathcal{S}(\varphi)) = \#(\sigma\cap \gamma_1)\cdot \nu(\gamma_1)_*\circ \Psi\circ \mathcal{S}(\varphi).
\]
In conclusion, we have 
\[
r\circ \Psi(\alpha) = k\cdot \#(\sigma\cap \gamma_1)\cdot \nu(\gamma_1)_*\circ \Psi\circ \mathcal{S}(\varphi),
\]
where $k$ equals the number of indices $i$ such that Cases (2) or (3) above hold. It is clear that $k>0$, since $i=1$ satisfies Case (2). Therefore, $r\circ \Psi(\alpha)\neq 0$, so $\Psi(\alpha)\neq 0$.

By the assumptions on $\sigma$ and Proposition \ref{prop_trivial_centralizer_in_Emb_dagger}, we know that the image of $(\mathscr{I}\circ\sigma)_*:\pi_1(\Sigma)\to\pi_1(\widetilde{Y})$ has a trivial centralizer in $\pi_1\Emb_\dagger(I,I\times \widetilde{Y})$.  Hence, the map $f$ satisfies the assumptions of Lemma \ref{lem_surface_hmtpy_closed}. 
Since $\Psi(\alpha)\neq 0$, we conclude that $f$ is not homotopic to $\varphi_*\circ f$ by Lemmas \ref{lem_surface_hmtpy_closed}, \ref{lem_closed_surface_basepoint_connected_sum_homotopy}, and Proposition \ref{prop_invariance_Psi_under_pi1_Y}. By Lemma \ref{lem_Embdagger_obstruct_isotopy}, $\varphi$ is not isotopic to $\id$ in $\Homeo_0(I\times\widetilde{Y})$.

This shows that the image of the composition \eqref{eqn_composition_from_cylinder_to_covering_non_fibered_case} is of infinite rank. As a consequence, the composition 
\[
\pi_0\Diff_{PI}(I\times Y_0)\xrightarrow{(\id_I\times \hat{\nu}(\gamma))_*} \pi_0\Diff_{PI}(I\times Y)
\to \pi_0\Homeo_{0}(I\times Y)
\]
is of infinite rank. Since the canonical map $\pi_0\Homeo_{0}(I\times Y)\to \pi_0\Homeo(I\times Y)$ is an injection, the desired theorem is proved.
\end{proof}

\begin{remark}
    It may be counterintuitive that the contributions to $r\circ \Psi(\alpha)$ by Cases (2) and (3) above have the same sign instead of canceling each other. This can be illustrated by Corollary \ref{cor_homeo_reverse_H1_on_Y0}: if we reverse all the orientations of $\gamma_i$, then the orientations of the normal bundles are also reversed, and the value of $\Psi(\alpha)$ would remain the same instead of having the reversed sign. 
\end{remark}

The next result shows that every closed, oriented, irreducible $3$--manifold $Y$ such that $b_1(Y)>0$ and $Y$ is not Seifert fibered satisfies the assumptions of Proposition \ref{prop_trivial_centralizer_in_Emb_dagger}. Therefore, the assumptions about $\pi_1(\widetilde{Y})$ in Theorem \ref{thm_main_irred_case_when_xi1_xi2_exist_in_pi1} are satisfied if $\widetilde{Y}$ is not Seifert fibered.
\begin{Proposition}
\label{prop_existence_xi1_xi2_non_fibered_case}
    Suppose $Y$ is a closed, oriented, irreducible $3$--manifold, such that $b_1(Y)>0$ and $Y$ is not Seifert fibered. Then there exist $\xi_1,\xi_2\in \pi_1(Y)$ such that 
    \begin{enumerate}
        \item $\xi_1$ is of infinite order, and for each positive integer $m$, the centralizer of $(\xi_1)^m$ in $\pi_1(Y)$ is the infinite cyclic group generated by $\xi_1$. 
        \item $\xi_2$ is not commutative with $\xi_1$, and $(\xi_2)^2 \neq 1$. 
    \end{enumerate} 
\end{Proposition}

\begin{proof}
If $Y$ is a closed hyperbolic manifold, then every non-trivial element in $\pi_1(Y)$ is of infinite order, and the centralizer of every non-trivial element is cyclic. Let $\xi_1\in \pi_1(Y)$ be an indivisible element, and let $\xi_2\in \pi_1(Y)$ be an element that is not commutative to $\xi_1$, then $\xi_1,\xi_2$ satisfy the desired properties. From now on, we will assume that $Y$ is not hyperbolic. Since we also assume $Y$ is irreducible and not Seifert fibered, $Y$ has a non-trivial JSJ decomposition.

Let $T_1,\dots,T_k\subset Y$ $(k\ge 1)$ be the JSJ tori of $Y$, and let $Y_1,\dots,Y_l$ be the connected components of $Y\setminus (T_1\cup\dots\cup T_k)$. Then $\pi_1(Y_i)\to \pi_1(Y)$ and $\pi_1(T_j)\to \pi_1(Y)$ are injective for all $i,j$, and the images of all these homomorphisms are proper subgroups of $\pi_1(Y)$. 

Let $p:Z\to Y$ be the universal cover of $Y$. Then every component of $p^{-1}(Y_i)$ and $p^{-1}(T_j)$ is simply connected. We construct a graph $\mathcal{T}$ as follows. Let the vertices of $\mathcal{T}$ be the components of $p^{-1}(Y_i)$, and the edges be the components of $p^{-1}(T_j)$. The two vertices of each edge are given by the two sides of the component of $p^{-1}(T_j)$. By the Seifert-van Kampen theorem, $\pi_1(Z)\cong \pi_1(\mathcal{T})$, so $\mathcal{T}$ is a tree. We endow a metric on $\mathcal{T}$ such that each edge is a geodesic with length $1$.

The group $\pi_1(Y)$ acts on $\mathcal{T}$ isometrically. Since $Y$ is an oriented manifold, we may assign an orientation to each $T_j$ and use them to assign an orientation to each edge of $\mathcal{T}$ so that the action of $\pi_1(Y)$ on $\mathcal{T}$ preserves the edge orientations. As a consequence, the action of $\pi_1(Y)$ on $\mathcal{T}$ admits a non-empty fixed point set if and only if there exists a fixed vertex. Since the image of each $\pi_1(Y_i)\to \pi_1(Y)$ is a proper subgroup of $\pi_1(Y)$, we know that the action of $\pi_1(Y)$ on $\mathcal{T}$ has no global fixed vertices. Therefore, the action has no global fixed points. Since $\pi_1(Y)$ is finitely generated, by \cite[I \S 6.5, Corollary 3]{serre2002trees}, we conclude that there exists an element in $\pi_1(Y)$ whose action on $\mathcal{T}$ has no fixed point. 

For $\xi\in \pi_1(Y)$, let 
\[
l(\xi) = \inf_{x\in\mathcal{T}} d(x,\xi(x)).
\]
Then $l(\xi)$ is an integer, and it equals zero if and only if the action of $\xi$ admits a fixed point. Let $\xi_1\in \pi_1(Y)$ be such that $l(\xi_1)>0$ and $l(\xi_1)$ achieves the minimum positive value.

By \cite[I \S 6.4]{serre2002trees}, we know that $l(\xi^k)=|k|\, l(\xi)$ for all $\xi\in \pi_1(Y)$ and  integers $k$. Therefore, the element $\xi_1$ constructed above is of infinite order, and the only cyclic subgroup of $\pi_1(Y)$ that contains $\xi_1$ is the infinite cyclic group generated by $\xi_1$.

By a result of Jaco--Shalen and by Johannson (\cite{jaco1979seifert, johannson2006homotopy}, see also \cite[Theorem 1.1]{friedl2011centralizers}), if $\xi\in \pi_1(Y)$ has a non-cyclic centralizer, then $\xi$ is conjugate to an element in $\pi_1(T_j)$ or $\pi_1(Y_i)$ with $Y_i$ Seifert fibered. This implies the action of $\xi$ on $\mathcal{T}$ has a fixed vertex.
Therefore, for each positive integer $m$, the centralizer of $(\xi_1)^m$ is a cyclic group. 

Since the centralizer group of $(\xi_1)^m$ contains $\xi_1$, it must be the cyclic group generated by $\xi_1$. This shows that $\xi_1$ satisfies Condition (1).

It is a standard result in $3$--manifold topology that the only closed oriented $3$--manifold with $\pi_1$ isomorphic to $\mathbb{Z}$ is $S^1\times S^2$ (it is straightforward to show that such a manifold must be prime and not aspherical, so it must be $S^1\times S^2$). Hence our assumptions imply $\pi_1(Y)\not\cong \mathbb{Z}$. Recall that we assume $b_1(Y)>0$, so $\pi_1(Y)$ is generated by elements with non-trivial fundamental classes in $H_1(Y;\bQ)$. Therefore, there exists $\xi_2\in \pi_1(Y)$, whose image in $H_1(Y;\bQ)$ is non-zero, such that $\xi_2$ is not commutative with $\xi_1$. Since the image of $\xi_2$ in $H_1(Y;\bQ)$ is non-zero, it is of infinite order.  Therefore, $\xi_1,\xi_2$ satisfy Condition (2). 
\end{proof}

Applying Proposition \ref{prop_existence_xi1_xi2_non_fibered_case} to $\widetilde{Y}$ and invoking Theorem \ref{thm_main_irred_case_when_xi1_xi2_exist_in_pi1}, we conclude that Theorem \ref{thm_main_irred_case} holds if $Y$ is closed and $\widetilde{Y}$ is not Seifert fibered. 

\begin{remark}
    The conclusion of Proposition \ref{prop_existence_xi1_xi2_non_fibered_case} does not hold if $Y$ is a closed Seifert fibered manifold. If $Y$ is Seifert fibered with orientable fibers, then the center of $\pi_1(Y)$ contains an isomorphic copy of $\mathbb{Z}$ generated by the image of the fiber. If $Y$ is Seifert fibered with non-orientable fibers, then $\pi_1(Y)$ has an index-two subgroup with a center that contains an isomorphic copy of $\mathbb{Z}$. In either case, it is straightforward to see that the elements $\xi_1,\xi_2$ described in Proposition  \ref{prop_existence_xi1_xi2_non_fibered_case} cannot exist. 
\end{remark}

\subsection{The case when $\widetilde{Y}$ is Seifert fibered}

Now we prove Theorem \ref{thm_main_irred_case} in the case when $\widetilde{Y}$ is Seifert fibered. The proof will be similar to the proof of Theorem \ref{thm_main_irred_case_when_xi1_xi2_exist_in_pi1}, except that we will use the $\U(1)$ action instead of the trivial centralizer groups to obstruct homotopies of surfaces. 

Let $B$ be the orbit space of $\widetilde{Y}$. Since $\pi_2(\widetilde{Y})=0$, we have $\chi_{\orb}(B)\le 0$. Therefore, there exists a covering map $\widetilde{B}\to B$, where $\widetilde{B}$ is a closed oriented surface with genus $\ge 1$. So $Y$ is covered by an $S^1$ bundle over a closed oriented surface with genus $\ge 1$. Since finite covering spaces of $S^1$ bundles over surfaces are always $S^1$ bundles over surfaces, after taking a further covering, we may obtain a \emph{normal} covering of $Y$ by an $S^1$ bundle over a closed oriented surface with genus $\ge 1$. Therefore, without loss of generality, we may assume that $\widetilde{Y}$ is an $S^1$ bundle over a closed oriented surface with positive genus.

\begin{Theorem}
Suppose $Y$ admits a finite normal cover $\widetilde{Y}$ which is diffeomorphic to an $S^1$ bundle over a closed oriented surface with positive genus.
Then there exists a smooth embedding $Y_0=S^1\times D^2\to Y$, such that the image of $\pi_0\Diff_{PI}(I\times Y_0)\to \pi_0\Homeo(I\times Y)$ is of infinite rank. 
\end{Theorem} 

\begin{proof}

Let $S^1\to\widetilde{Y}\to B$ be the fiber bundle structure of $\widetilde{Y}$. By the assumptions, the fibers are oriented, so we may identify $\widetilde{Y}$ with a $\U(1)$ bundle. 

Let $\gamma:S^1\to \widetilde{Y}$ be a smooth embedding, such that its projection image in $B$ has a primitive fundamental class in $H_1(B;\mathbb{Z})$. After perturbation, we may assume that the projection image of $\gamma$ in $Y$ is embedded, so $\gamma$ is disjoint from its images under the deck transformation group. Let $\gamma_1,\dots,\gamma_k$ be the images of $\gamma$ under deck transformations, where $\gamma_1=\gamma$. Let $\nu(\gamma_i):Y_0=S^1\times D^2 \to \widetilde{Y}$ be an orientation-preserving embedding whose core circle is $\gamma_i$. We may require that the maps $\nu(\gamma_i)$ are related to each other by deck transformations, and that the images of $\nu(\gamma_i)$ are disjoint from each other. 

Let $\varphi_1,\varphi_2,\dots\in \Diff_{PI}(I\times Y_0)$ be given by Theorem \ref{thm_BG_construction_S1xD3}. Let $\tilde{\varphi}_{i,j}$ be the extension of $\varphi_j$ to $I\times Y$ via $\id_{I}\times \nu(\gamma_i)$. Then for a fixed value of $j$, the diffeomorphisms $\tilde{\varphi}_{i,j}$ for different values of $i$ are commutative to each other because they have disjoint supports. Let $\tilde{\varphi}_j = \prod_i \tilde{\varphi}_{i,j}$. Then 
the isotopy class of $\tilde{\varphi}_j$ is in the image of 
\begin{align}
        & \pi_0\Diff_{PI}(I\times Y_0)\xrightarrow{(\id_I\times \hat{\nu}(\gamma))_*}  \pi_0\Diff_{PI}(I\times Y)\nonumber\\
& \quad \to \pi_0\Homeo_{0}(I\times Y)\xrightarrow{(\id_I\times p)^*} \pi_0\Homeo_{0}(I\times \widetilde{Y}).
\label{eqn_composition_from_cylinder_to_covering_fibered_case}
\end{align}

We show that $\tilde{\varphi}_j$ are linearly independent over $\mathbb{Z}$ in $\pi_0\Homeo_{0}(I\times \widetilde{Y})$.
Let $0\neq (c_j)\in \oplus_\infty \mathbb{Z}$, let $\varphi = \prod_j \tilde{\varphi}_j^{c_j}\in \Homeo_{PI}(I\times \widetilde{Y})$. We show that $\varphi$ is not isotopic to the identity in $\Homeo_{0}(I\times \widetilde{Y})$. 

Let $\mu:S^1\to B$ be an immersed curve that has a non-zero intersection number with the projection image of $\gamma_1$ in $B$. The pull-back of $\widetilde{Y}$ (as an $S^1$ bundle) to $\mu$ is a trivial $S^1$ bundle. Let $\sigma:S^1\times S^1\to \widetilde{Y}$ be the bundle map that lifts $\mu$. Let $f=\mathscr{I}\circ \sigma:S^1\times S^1\to \Emb_\dagger(I,I\times\widetilde{Y})$, let $\varphi_*: \Emb_\dagger(I,I\times\widetilde{Y})\to \Emb_\dagger(I,I\times\widetilde{Y})$ be the homeomorphism induced by $\varphi$. We show that $f$ and $\varphi_*\circ f$ are not homotopic. 

By Proposition \ref{prop_interior_connected_sum_formula_on_surface_homotopy_multi_curves}, $\varphi_*\circ f$ is obtained from $f$ by an interior connected sum with $\alpha\in \pi_2\Emb_\dagger(I,I\times\widetilde{Y})$, where
\[
\Psi(\alpha) = \sum_i \#(\sigma\cap \gamma_i)\cdot \nu(\gamma_i)_*\circ \Psi\circ \mathcal{S}(\varphi).
\]
The same argument as in Theorem \ref{thm_main_irred_case_when_xi1_xi2_exist_in_pi1} shows that $\Psi(\alpha)\neq 0$.

Note that the map $f$ satisfies the assumptions of Lemma \ref{lem_map_from_torus_degenerate_after_m}. This is because $\widetilde{Y}$ admits an $\U(1)$ action $m:\U(1)\times \widetilde{Y}\to \widetilde{Y}$ from its bundle structure. 
This induces $\U(1)$ actions on $I\times \widetilde{Y}$ and on $\Emb_\dagger(I,I\times \widetilde{Y})$, which we also denote by $m$. 
By the construction of $\sigma:S^1\times S^1\to \widetilde{Y}$, we know that $S^1\times S^1$ is endowed with the structure of a trivial $\U(1)$ bundle over $S^1$, and $\sigma$ is $\U(1)$--equivariant. 
As a result, $m\circledast f$ satisfies the desired properties.

Since $\Psi(\alpha)\neq 0$, we conclude that $f$ is not homotopic to $\varphi_*\circ f$ by Lemma \ref{lem_map_from_torus_degenerate_after_m} and Proposition \ref{prop_invariance_Psi_under_pi1_Y}. Therefore, by Lemma \ref{lem_Embdagger_obstruct_isotopy}, $\varphi$ is non-trivial in $\pi_0\Homeo_{0}(I\times \widetilde{Y})$.

This shows that the image of the composition \eqref{eqn_composition_from_cylinder_to_covering_fibered_case} is of infinite rank. As a consequence, the composition $\pi_0\Diff_{PI}(I\times Y_0)\xrightarrow{\hat{\nu}(\gamma)_*} \pi_0\Diff_{PI}(I\times Y)
\to \pi_0\Homeo_{0}(I\times Y)$ is of infinite rank. So the desired result is proved.
\end{proof}

This concludes the proof of Theorem \ref{thm_main_irred_case} when $Y$ is closed and $\widetilde{Y}$ is Seifert fibered. 

\part{The case when $Y$ is reducible after filling $D^3$'s}\label{part_red_case}
Similar to the notation in Part \ref{part_irred_case}, we let $\hat Y$ be the manifold obtained by filling each boundary component of $Y$ that is diffeomorphic to $S^2$ with a copy of $D^3$. The following theorem is the main result of Part \ref{part_red_case}.

\begin{alphthm}
\label{thm_main_red_case}
Suppose $Y$ is a compact, connected, reducible, oriented $3$--manifold, such that $\pi_1(Y)$ is infinite and $\partial Y$ contains no connected components diffeomorphic to $S^2$. Then there exists a smooth embedding $i:X_0 \to  I\times Y$, such that the image of the composition map   
\[
        \pi_0 \Diff_{PI}(X_0)\xrightarrow{i_*} \pi_0 \Diff_{PI}(I\times Y) \to \pi_0\Homeo (I\times Y,I\times \partial Y)
    \]
is of infinite rank. Here, we set $X_0$ to be $S^1\times D^3$ if $Y=S^1\times S^2$ or $\mathbb{RP}^3\#\mathbb{RP}^3$ and we set $X_0$ to be $(S^2\times D^2)\natural(S^1\times D^3)$ otherwise. 
\end{alphthm}

Theorem \ref{thm_main_red_case} implies Theorems \ref{thm_main} and \ref{thm_main2} in the case when $\hat Y$ is reducible. This follows from the same argument after the statement of Theorem \ref{thm_main_irred_case} by replacing $I\times Y_0$ with $X_0$. The rest of Part \ref{part_red_case} is devoted to the proof of Theorem \ref{thm_main_red_case}.

\section{The case when $Y=S^1\times S^2$}

This section proves Theorem \ref{thm_main_red_case} when $Y=S^1\times S^2$. Suppose $Y=S^1\times S^2$. Let $\gamma:S^1\to Y$ be an embedding that induces an isomorphism on $H_1$. Let $\nu(\gamma):S^1\times D^3\to I\times \widetilde{Y}$ be an orientation-preserving embedding with core circle $\gamma$. Let $\varphi_j\in \Diff_{PI}(S^1\times D^3)\cong \Diff_{PI}(I\times S^1\times D^2)$ be the sequence of diffeomorphisms given by Theorem \ref{thm_BG_construction_S1xD3}. Let $\nu(\gamma)_*(\varphi_j)$ be the extension of $\varphi_j$ to $I\times Y$ via $\nu(\gamma)$. We will show that the images of $\nu(\gamma)_*(\varphi_j)$ are linearly independent over $\mathbb{Z}$ in $\pi_0\Homeo(I\times Y)$. 

We start with the following lemmas. First, note that by the isotopy extension properties, we have a Serre fibration
\[
\Homeo_\partial(S^1\times D^3)\hookrightarrow \Homeo(S^1\times D^3)\to \Homeo(S^1\times S^2)
\]
\begin{Lemma}
\label{lem_S1xD3_boundary_Homeo_pi1}
    The restriction map $\Homeo(S^1\times D^3)\to \Homeo(S^1\times S^2)$ induces a surjection on the loop space based at $\id$.
\end{Lemma}
\begin{proof}
    The homotopy type of $\Homeo(S^1\times S^2)$ was computed by Hatcher \cite{hatcher1981diffeomorphism}. In particular, it is known that $\pi_1\Homeo(S^1\times S^2) \cong \bZ\oplus \bZ/2$, and it is generated by a rotation on the $S^1$ factor and a rotation on the $S^2$ factor by the non-trivial element in $\pi_1(\SO(3))$. Both loops can be lifted to $\Homeo(S^1\times D^3)$. So the map $\Homeo(S^1\times D^3)\to \Homeo(S^1\times S^2)$ induces a surjection on $\pi_1$. Since the map is a Serre fibration, it also induces a surjection on the loop space.
\end{proof}
Let $\kappa:I\times Y\to S^1\times D^3$ be an embedding whose image is a closed collar neighborhood of $\partial (S^1\times D^3)$. 

\begin{Lemma}
\label{lem_S1xS2_push_to_S1xD3}
Suppose $\varphi\in \Homeo_\partial(I\times Y)$, let $\kappa_*\varphi$ denote the extension of $\varphi$ to $S^1\times D^3$ via $\kappa$. If $[\varphi]$ is trivial in $\pi_0\Homeo(I\times Y)$, then $[\kappa_*\varphi]$ is trivial in $\Homeo(S^1\times D^3)$.
\end{Lemma}
\begin{proof}
Without loss of generality, assume $\kappa(\{1\}\times Y) = \partial (S^1\times D^3)$. 
    Suppose $F:I\times (I\times Y)\to I\times Y$ is an isotopy from $\varphi$ to $\id$ in $\Homeo(I\times Y)$. Then by Lemma \ref{lem_S1xD3_boundary_Homeo_pi1}, there exists an isotopy $G:I\times (S^1\times D^3)\to S^1\times D^3$ from $\id$ to $\id$, such that $G|_{I\times \partial(S^1\times D^3)} = F|_{I\times \{0\}\times Y}$. Gluing $G$ and $F$ yields the isotopy from $\kappa_*\varphi$ to $\id$ in $\Homeo(S^1\times D^3)$.
\end{proof}

Recall that we denote $Y_0=S^1\times D^2$, so $I\times Y_0\cong S^1\times D^3$.
\begin{Lemma}
\label{lem_w3_obstruct_isotopy_moving_boundary}
    Suppose $\varphi\in \Homeo_\partial(I\times Y_0) \cong \Homeo_\partial(S^1\times D^3)$ satisfies $\Psi\circ \mathcal{S}([\varphi])\neq 0$, then $[\varphi]$ is non-trivial in $\pi_0\Homeo(S^1\times D^3)$. 
\end{Lemma}
\begin{proof}
    We have a Serre fibration 
    \[
    \Homeo_\partial(S^1\times D^3) \hookrightarrow \Homeo(S^1\times D^3)\to \Homeo(S^1\times S^2),
    \]
    therefore an exact sequence 
    \begin{align*}
    &\pi_1 \Homeo(S^1\times D^3) \to \pi_1 \Homeo(S^1\times S^2) \\
    &\quad \to \pi_0  \Homeo_\partial(S^1\times D^3) \to \pi_0 \Homeo(S^1\times D^3).
    \end{align*}
    By Lemma \ref{lem_S1xD3_boundary_Homeo_pi1}, $\pi_1 \Homeo(S^1\times D^3) \to \pi_1 \Homeo(S^1\times S^2)$ is surjective, so $\pi_0  \Homeo_\partial(S^1\times D^3) \to \pi_0 \Homeo(S^1\times D^3)$ is injective, and hence the result follows. 
\end{proof}

Now we can finish the proof of Theorem \ref{thm_main_red_case} when $Y=S^1\times S^2$.
\begin{proof}
Let $0\neq (c_k)\in \oplus_\infty \mathbb{Z}$, and let $\hat \varphi = \prod_k \varphi_k^{c_k}$, let 
$\varphi = \prod_{k} \nu(\gamma)_*(\varphi_k)^{c_k} =\nu(\gamma)_*(\hat\varphi).$
We show that $\varphi$ is non-trivial in $\pi_0\Homeo(I\times Y)$. Note that $\kappa\circ \gamma$ is isotopic to the core circle of $S^1\times D^3$, so we may choose $\nu(\gamma)$ so that $\kappa\circ \nu(\gamma)$ is isotopic to the standard embedding of $S^1\times D^3$ to a neighborhood of the core circle in $S^1\times D^3$. Therefore, $\kappa_*\varphi$ is isotopic in $\Homeo_\partial(S^1\times D^3) = \Homeo_\partial(I\times Y_0)$ to $\hat\varphi$. By Theorem \ref{thm_BG_construction_S1xD3}, we know that $\Psi\circ\mathcal{S}([\kappa_\ast\varphi])\neq 0$. So by Lemma \ref{lem_S1xS2_push_to_S1xD3} and Lemma \ref{lem_w3_obstruct_isotopy_moving_boundary}, we know that $\varphi$ is non-trivial in $\pi_0\Homeo(I\times Y)$. 
\end{proof}

\section{The case when $Y=\mathbb{RP}^3\#\mathbb{RP}^3$}
This section proves Theorem \ref{thm_main_red_case} when $Y=\mathbb{RP}^3\#\mathbb{RP}^3$. Suppose $Y=\mathbb{RP}^3\#\mathbb{RP}^3$, then there is a double covering map $p:S^1\times S^2\to Y$. Let $\widetilde{Y}=S^1\times S^2$. The non-trivial deck transformation $\tau:\widetilde{Y}\to \widetilde{Y}$ is given by the product of a reflection on $S^1$ and the antipodal map on $S^2$. 

Let $\hat\tau$ be the product of a reflection on $S^1$ and the map $x\mapsto -x$ on $D^3$. Then $\hat\tau$ is an involution on $S^1\times D^3$. There exists an embedding $\kappa:I\times \widetilde{Y}\to S^1\times D^3$, whose image is a closed collar neiborhood of $\partial (S^1\times D^3)$, such that $\hat\tau\circ \kappa = \kappa\circ (\id_I\times \tau)$. 

The non-trivial deck transformation on $S^1\times S^2$ is not homotopic to the identity. Therefore, by Lemma \ref{lem_pull_back_diff_PI}, the pull-back map 
\[
\pi_0\Homeo_{0}(I\times Y)\xrightarrow{(\id_I\times p)^*}\pi_0\Homeo_{0}(I\times \widetilde{Y})
\]
is well-defined. 

Let $\gamma:S^1\to I\times \widetilde{Y}$ be a smooth embedding whose image represents a generator of $\pi_1(I\times \widetilde{Y})$. Perturb $\gamma$ so that its image is disjoint from $\tau\circ \gamma$. Write $\gamma_1=\gamma$, $\gamma_2 = (\id_I\times\tau)\circ \gamma$. Let $\nu(\gamma_i):S^1\times D^3\to I\times \widetilde{Y}$ be an orientation-preserving embedding with core circle $\gamma_i$. We also require that $\nu(\gamma_i)$ $(i=1,2)$ are disjoint, and $\nu(\gamma_2) = (\id_I\times \tau)\circ \nu(\gamma_1)$. Define $\hat\nu(\gamma)$ to be the embedding of $S^1\times D^3$ in $I\times Y$ given by 
\[
\hat\nu(\gamma)=(\id_I\times p)\circ \nu(\gamma_1) = (\id_I\times p)\circ \nu(\gamma_2).
\]

Let $\varphi_j\in \Diff_{PI}(S^1\times D^3)\cong \Diff_{PI}(I\times S^1\times D^2)$ be the sequence of diffeomorphisms given by Theorem \ref{thm_BG_construction_S1xD3}. Let $\tilde{\varphi}_{i,j}$ be the extension of $\varphi_j$ to $I\times \widetilde{Y}$ via $\nu(\gamma_i)$. Let $\tilde{\varphi}_j = \tilde{\varphi}_{1,j}\circ \tilde{\varphi}_{2,j}$. Then the isotopy class of $\tilde{\varphi}_j$ is in the image of 
\begin{align}
&\pi_0\Diff_{PI}(S^1\times D^3) \xrightarrow{\hat{\nu}(\gamma)_*}\pi_0\Diff_{PI}(I\times Y)\nonumber \\
&\quad \to \pi_0\Homeo_{0}(I\times Y)\xrightarrow{(\id_I\times p)^*}\pi_0\Homeo_{0}(I\times \widetilde{Y}).
\label{eqn_composition_from_cylinder_S1xS2}
\end{align}

We show that $\tilde{\varphi_j}$ are linearly independent in $\pi_0\Homeo_{0}(I\times \widetilde{Y})$. 
Consider the extension map
\[
\kappa_*:\pi_0\Homeo_{PI}(I\times \widetilde{Y}) \to \pi_0\Homeo_{PI}(S^1\times D^3).
\]
Then $\kappa_*(\tilde{\varphi}_j) = \kappa_*(\tilde{\varphi}_{1,j})\circ \kappa_*(\tilde{\varphi}_{2,j})$, and we have $\kappa_*(\tilde{\varphi}_{2,j})=\tau \circ  \kappa_*(\tilde{\varphi}_{1,j})\circ \tau^{-1}$. Now we invoke Corollary \ref{cor_homeo_reverse_H1_on_Y0}. Note that we can choose $\hat\iota$ in Corollary \ref{cor_homeo_reverse_H1_on_Y0} so that $\tau\circ  \kappa_*(\tilde{\varphi}_{1,j})\circ \tau^{-1}$ is isotopic to $\hat{\iota}\circ  \kappa_*(\tilde{\varphi}_{1,j})\circ \hat{\iota}^{-1}$. So we have 
\[
\Psi\circ \mathcal{S}( \kappa_*(\tilde{\varphi}_{1,j})) = \Psi\circ \mathcal{S}( \kappa_*(\tilde{\varphi}_{2,j})),
\]
and hence
\[
\Psi\circ \mathcal{S}( \kappa_*(\tilde{\varphi}_{j})) = 2 \cdot \Psi\circ \mathcal{S}( \kappa_*(\tilde{\varphi}_{1,j})).
\]
Therefore, $\kappa_*(\tilde{\varphi}_j)$ have linearly independent images under $\Psi\circ \mathcal{S}$, so $\tilde{\varphi}_j$ are linearly independent in $\pi_0\Homeo_{0}(I\times \widetilde{Y}) $ by Lemma \ref{lem_S1xS2_push_to_S1xD3} and Lemma \ref{lem_w3_obstruct_isotopy_moving_boundary}.

In summary, we have proved that the image of \eqref{eqn_composition_from_cylinder_S1xS2} is of infinite rank. Therefore, the composition 
\[
\pi_0\Diff_{PI}(S^1\times D^3) \xrightarrow{\hat{\nu}(\gamma)_*}\pi_0\Diff_{PI}(I\times Y) \to \pi_0\Homeo(I\times Y)
\]
is of infinite rank.

\section{The case when $Y$ is $ Y_1\#Y_2$ with $Y_2\not\cong \mathbb{RP}^3$} \label{connected_sum_case}
Now we assume $Y$ satisfies the assumptions of Theorem \ref{thm_main_red_case} and $Y\not\cong S^1\times S^2$ or $\mathbb{RP}^3\#\mathbb{RP}^3$. Then we can write $Y=Y_1\#Y_2$ with $Y_1\not\cong S^3$, and $Y_2\not\cong S^3$ or $\mathbb{RP}^3$. Since no component of $\partial Y$ is diffeomorphic to $S^2$, we have $\pi_1(Y_1)\not\cong \{1\}$, and $\pi_1(Y_2)\not\cong \{1\}$ or $\mathbb{Z}/2$. In this case, we will construct a sequence of barbell diffeomorphisms on $I\times Y$, and show that they are linearly independent in $\pi_0\Homeo(I\times Y,I\times \partial Y)$ by considering the actions on the homotopy classes of curves in $\Emb_\dagger(I,I\times Y)$. 

\begin{Lemma}
    There exists $\alpha\in \pi_1(Y_2)$ such that $\alpha^2\neq 1$.
\end{Lemma}

\begin{proof}
    Assume every element of $\pi_1(Y_2)$ is of order $2$, then for all $\alpha,\beta\in \pi_1(Y_2)$, we have 
    \[
\alpha\beta = (\alpha\beta)^{-1} = \beta^{-1}\alpha^{-1} = \beta\alpha,
    \]
    so $\pi_1(Y_2)$ is abelian. It is known that the only abelian $3$--manifold groups are the cyclic groups, $\mathbb{Z}^2$, $\mathbb{Z}^3$, and $\mathbb{Z}\oplus \mathbb{Z}/2$ (see, for example, \cite[Table 1.2]{aschenbrenner20153manifold}). Since $\pi_1(Y_2)\not\cong\{1\}$ or $\bZ/2$, the desired result is proved.
\end{proof}

Let $S^2\cong S\subset Y$ be the attaching sphere in the connected sum decomposition. Let $\nu(S)$ denote a closed tubular neighborhood of $S$ in $Y$. We fix an orientation-preserving diffeomorphism $\rho:[1,2]\times S\cong \nu(S)$
and define 
\[
\widetilde{\rho}: I\times [1,2]\times S\to I\times Y,\quad  \widetilde{\rho}(t,s,x):=(t,\rho(s,x)).
\]
Fix two points $b\neq b'\in S$ and set  $b$ as the base point. We write $Y$ as 
\[
Y = Y_1^{\circ}\cup_{\rho(\{1\}\times S)} \big(\nu(S)\big) \cup_{\rho(\{2\}\times S)}Y_2^{\circ},
\]
where $Y_1^{\circ}$, $Y_2^{\circ}$ are the closures of the two components of $Y\setminus \nu(S)$. 

Pick nontrivial elements $\alpha_{i}\in\pi_{1}(Y_{i},b)$ $(i=1,2)$, and we require that $\alpha_{2}$ is not $2$--torsion. Let $\alpha=\alpha_1\cdot \alpha_2\in \pi_{1}(Y,b)$. 
We represent $\alpha$ by an embedded loop $\hat{\alpha}:S^1\cong \mathbb{R}/4\mathbb{Z}\hookrightarrow Y$ that satisfies the following conditions:
\begin{enumerate}
      \item $\hat{\alpha}(t)=\rho(t,b)$ for all $t\in [1,2]$.
    \item $\hat{\alpha}(t)=\rho(4-t,b')$ for all $t\in [3,4]$. 
    \item $\hat{\alpha}([0,1])\subset Y^{\circ}_{1}$ and represents the element 
    \[\alpha_{1}\in \pi_{1}(Y^{\circ }_1,\rho(\{1\}\times S))\cong \pi_{1}(Y_1).\]
    \item $\hat{\alpha}([2,3])\subset Y^{\circ}_{2}$ and represents the element 
    \[\alpha_{2}\in \pi_{1}(Y^{\circ }_2,\rho(\{2\}\times S))\cong \pi_{1}(Y_2).\]
\end{enumerate}
Here, $\pi_1(Y_{i}^{\circ},\rho(\{i\}\times S))$ denotes the relative homotopy groups. 

Recall that $X_0=(S^2\times D^2)\natural(S^1\times D^3)$.
We consider the self-referential barbell $\mathcal{B}\subset X_0$, defined as follows: Let $S_0=S^2\times \{0\}\subset S^2\times D^2$. Pick an embedding \[D=D^3\hookrightarrow (S^2\times \mathrm{int}(D^2))\setminus S_0\subset X_0.\] Let $S_1=\partial D^3$. Take an arc $\gamma: I\to X_0$ that satisfies the following conditions: 
\begin{enumerate}
    \item $\gamma(0)\in S_0$,  $\gamma(1)\in S_1,\  \gamma([0,\frac{1}{2}])\subset S^2\times D^2$;
\item $\gamma(0,1)\cap S_0=\emptyset$ and $\gamma(0,1)\pitchfork D=\gamma(\frac{1}{2})\in \mathrm{int}(D)$;
\item $[\gamma]$ represents a generator $g$ in $\mathbb{Z}\cong \pi_{1}(X_0,S^2\times D^2)$.
\end{enumerate}
Then the regular neighborhood of $S_0\cup \mathrm{Im}(\gamma)\cup S_1$ is a barbell, which we denote by $\mathcal{B}$. We let $\psi_1\in \Diff_{\partial}(X_0)$ be the barbell diffeomorphism associated with $\mathcal{B}$. Note that for every integer $k>1$, we have a self-embedding $i_{k}:X_0\hookrightarrow X_0$ that takes $S_0$ to itself and satisfies $(i_{k})_*(g)=g^{k}\in \pi_{1}(X_0)$. By extending the diffeomorphism $i_{k}\circ \psi_{1}\circ i^{-1}_{k}\in \Diff_{\partial}(\mathrm{Im}(i_{k}))$ with the identity, we obtain a diffeomorphism $\psi_{k}\in \Diff_{\partial}(X_0)$. Since $S_1$ bounds a embedded $D^3$, by \cite[Proposition 2.6]{budney2025automorphism}, we have $\psi_{k}\in \Diff_{PI}(X_0)$. 

The next step is to turn $\psi_{k}$ into a diffeomorphism on $X$. We pick a smoothly embedded loop $\xi:I\to I\times Y$ that satisfies the following conditions:
\begin{enumerate}
    \item $\xi(0)=\xi(1)\in \{1/2\}\times (S\setminus \{b,b'\}).$
    \item $\Ima\xi \cap (I\times \Ima\hat{\alpha})=\emptyset$
    \item $[\xi]=\alpha\in \pi_{1}(I\times Y, \{1/2\}\times S)\cong \pi_{1}(Y)$.
\end{enumerate}
We take a closed tubular neighborhood $\nu(\{\frac{1}{2}\}\times S)$ of $\{\frac{1}{2}\}\times S\subset I\times Y$ that is contained in $\Ima\widetilde{\rho}$, and take a closed tubular neighborhood $\nu(\xi)$ of $\Ima(\xi)$ that does not intersect $I\times \Ima\hat{\alpha}$. For $x\in S^2$, we let $D_{x}=\{x\}\times D^2\subset S^2\times D^2$ and let $D'_{x}:=\widetilde{\rho}([0,1]\times [1,2]\times \{x\})\subset X$.
Then there exists an embedding 
\[
i: X_0\xrightarrow{\cong} \nu(\{1/2\}\times S)\cup \nu(\Ima\xi) \subset X
\]
such that 
\begin{equation}\label{eq: image of alpha intersect support of phi}
(\Ima i)\cap (I\times \Ima \hat{\alpha})=i(D_{b}\sqcup D_{b'})= D'_{b}\sqcup D'_{b'}.
\end{equation}

By extending $i\circ \psi_{k}\circ i^{-1}\in \Diff_{PI}(\Ima i)$ with the identity map, we obtain a diffeomorphism $\varphi_{k}\in \Diff_{PI}(I\times Y)$ for every $k\geq 1$. We will prove Theorem \ref{thm_main_red_case} by showing that $\{\varphi_{k}\}_{k\geq 1}$ are linearly independent in $\pi_0\Homeo(I\times Y, I\times \partial Y)$. For this, let $0\neq (c_1,\dots,c_n)\in \bZ^n$ and let $\varphi=\prod_k \varphi^{c_k}_{k}$. Then the diffeomorphism $\varphi$ induces a homeomorphism
\[
\varphi_*:\Emb_{\dagger}(I,I\times Y)\to \Emb_{\dagger}(I,I\times Y).
\]
\begin{Proposition}\label{prop: loop nonisotopic}
The loops 
\[
S^1\xrightarrow{\hat{\alpha}} Y\xrightarrow{\mathscr{I}}\Emb_{\dagger}(I,I\times Y),
\]
and 
\[
S^1\xrightarrow{\hat{\alpha}} Y\xrightarrow{\mathscr{I}}\Emb_{\dagger}(I,I\times Y)\xrightarrow{\varphi_*}\Emb_{\dagger}(I,I\times Y)
\]
are not freely homotopic.
\end{Proposition}
Proposition \ref{prop: loop nonisotopic} implies that the map $\varphi_*\circ\mathscr{I}$ is not homotopic to $\mathscr{I}$, and hence by Lemma \ref{lem_Embdagger_obstruct_isotopy}, $\varphi$ is not isotopic to the identity in $\Homeo(I\times Y,I\times \partial Y)$.  So Proposition \ref{prop: loop nonisotopic} implies Theorem \ref{thm_main_red_case}.

The rest of this section is devoted to the proof of Proposition \ref{prop: loop nonisotopic}.  Note that both loops $\mathscr{I}\circ \hat{\alpha}$ and $\varphi_*\circ \mathscr{I}\circ \hat{\alpha}$ take the base point of $S^1$ to the base point $\mathscr{I}(b)$ of $\Emb_\dagger(I,I\times Y)$, so they represent elements 
$[\mathscr{I}\circ \hat{\alpha}]$ and $ [\varphi_*\circ \mathscr{I}\circ \hat{\alpha}]$ in $\pi_{1}\Emb_{\dagger}(I,I\times Y)$.
Since $[\hat{\alpha}]=\alpha=\alpha_1\cdot \alpha_2\in \pi_{1}(Y)$, we have 
\begin{equation}\label{eq: i_alpha}
[\mathscr{I}\circ \hat{\alpha}]=\mathscr{I}_*(\alpha_1)\cdot \mathscr{I}_*(\alpha_2) \in \pi_{1}\Emb_{\dagger}(I,I\times Y).
\end{equation}

To compute $[\varphi_*\circ \mathscr{I}\circ \hat{\alpha}]$, we need some preparations. First, we note the isomorphism 
\[
\pi_{1}(\Emb_{\dagger}(I,I\times Y),\mathscr{I}(\nu(S)))\cong \pi_{1}\Emb_{\dagger}(I,I\times Y)
\]
because $\mathscr{I}(\nu(S))\cong I\times S^2$ is simply-connected. 
For $[a,b]\subset [0,4]$, we use $\widetilde{\alpha}_{[a,b]}$ to denote the path \[\varphi_{*}\circ \mathscr{I}\circ \hat{\alpha}|_{[a,b]}:[a,b]\to \Emb_{\dagger}(I,I\times Y).\] 
Then we have 
\begin{equation}\label{eq: decomposition of alpha}
\varphi_{*}\circ \mathscr{I}\circ \hat{\alpha}=\widetilde{\alpha}_{[0,1]}\cdot \widetilde{\alpha}_{[1,2]} \cdot \widetilde{\alpha}_{[2,3]}\cdot  \widetilde{\alpha}_{[3,4]}.    
\end{equation}
Since $\varphi$ is supported in the image of $i:X_0\to X$, by (\ref{eq: image of alpha intersect support of phi}), we have 
\[
[\widetilde{\alpha}_{[0,1]}]=[\mathscr{I}\circ \hat{\alpha}|_{[0,1]}]=\mathscr{I}_{*}([\alpha_1])\in \pi_{1}(\Emb_{\dagger}(I,I\times Y),\mathscr{I}(\nu(S)))
\]
and
\[
[\widetilde{\alpha}_{[2,3]}]=[\mathscr{I}\circ \hat{\alpha}|_{[2,3]}]=\mathscr{I}_{*}([\alpha_2])\in \pi_{1}(\Emb_{\dagger}(I,I\times Y),\mathscr{I}(\nu(S))).
\]
Recall from Section \ref{subsec_Dax_iso} that $\pi^{D}_{1}\Emb_{\partial}(I, I\times Y)$ denotes the kernel of the map $\Emb_{\partial}(I, I\times Y)\to \Map_{\partial}(I, I\times Y) \cong \pi_2(Y)$. 

\begin{Definition}
Define $F$ to be the composition map
\[
\pi^{D}_{1}\Emb_{\partial}(I, I\times Y)\hookrightarrow \pi_{1}\Emb_{\dagger}(I, I\times Y)\cong \pi_{1}(\Emb_{\dagger}(I, I\times Y),\mathscr{I}(\nu(S))).
\]
\end{Definition}

Let \[\beta_{0}:=\sum^{n}_{k=1}c_{k}(\alpha^{k}+\alpha^{-k})\in \bZ[\pi_1(Y)\setminus\{1\}]\cong \pi^{D}_{1}\Emb_{\dagger}(I,I\times Y).\]
\begin{Proposition}\label{prop: calculation of alpha} 
We have
 $[\widetilde{\alpha}_{[1,2]}]=F(\beta_0)$, $[\widetilde{\alpha}_{[3,4]}]=-F(\beta_0)$. 
\end{Proposition} 
\begin{proof}
Consider the diffeomorphism 
$\psi:=\prod_k\psi^{c_k}_{k}\in \Diff_{PI}(X_0).$
Note that the family of arcs \[\{\mathscr{I}\circ \hat{\alpha}(s)\}_{s\in [1,2]}\subset \Emb_{\dagger}(I,I\times Y).\]  is a scanning family for the disk $D'_{b}\hookrightarrow X$.
By \eqref{eq: image of alpha intersect support of phi}, we have  $i^{-1}(D'_{b})=D_{b}$. By the naturality of the Dax invariant, $[\widetilde{\alpha}_{[1,2]}]\in \pi_{1}(\Emb_{\dagger}(I,I\times Y),\mathscr{I}(S))$ equals the image of the relative Dax invariant 
\[
\mathrm{Dax}(D_{b},\psi(D_{b}))\in \pi^{D}_{1}(\Emb_{\partial}(I,X_0))\cong \mathbb{Z}[\pi_{1}(X_0)\setminus 1]
\]
under the composition 
\begin{equation}\label{eq: maps between embedding spaces}
\begin{split}
\pi^{D}_{1}(\Emb_{\partial}(I,X_0))&\xrightarrow{i_*}\pi^{D}_{1}(\Emb_{\partial}(I,X))\\
&\hookrightarrow \pi_{1}(\Emb_{\dagger}(I,X))\cong \pi_{1}(\Emb_{\dagger}(I,X),\mathscr{I}(S)). 
\end{split}
\end{equation}
Next, we compute the Dax invariant. 
By \cite[Theorem 4.9]{gabai2021self}, we have $\mathrm{Dax}(D_{b},\psi_1(D_{b}))=g+g^{-1}$, 
where $g$ is the standard generator of $\pi_{1}(X_0)\cong \mathbb{Z}$. Since $\psi_{k}$ is extended from $i_{k}\circ \psi_{1}\circ i^{-1}_{k}$, and the map $(i_{k})_*:\pi_{1}(X_0)\to \pi_{1}(X_0)$ sends $g$ to $g^{k}$, we have $\mathrm{Dax}(D_{b},\psi_k(D_{b}))=g^k+g^{-k}$. 

Given $\psi,\psi'\in \Diff_{PI}(X_0)$, by the additivity and the naturality of the Dax invariant, we have
\begin{align*}
\mathrm{Dax}(D_{b},(\psi\circ \psi')(D_{b}))&=\mathrm{Dax}(D_{b},\psi(D_{b}))+\mathrm{Dax}(\psi(D_{b}),(\psi\circ\psi')(D_{b}))\\
& = \mathrm{Dax}(D_{b},\psi(D_{b}))+\psi_*\mathrm{Dax}(D_{b},\psi'(D_{b}))\\
& = \mathrm{Dax}(D_{b},\psi(D_{b}))+\mathrm{Dax}(D_{b},\psi'(D_{b})).
\end{align*}
The last equation above uses the fact that $\psi$ induces the identity map on $\pi_1(X_0)$. 
This implies 
\begin{equation}\label{eq: Dax of psi}
\mathrm{Dax}(D_{b},\psi(D_{b}))=\sum^{n}_{i=1}c_{k}(g^{k}+g^{-k}).    
\end{equation} Under (\ref{eq: maps between embedding spaces}), the element $\sum^{n}_{i=1}c_{k}(g^{k}+g^{-k})$ is sent to $F(\beta_0)$. This finishes the proof of $[\widetilde{\alpha}_{[1,2]}]=F(\beta_0)$. The proof of $[\widetilde{\alpha}_{[3,4]}]=-F(\beta_0)$ is similar.
\end{proof}

\begin{Corollary} Let $\beta=F(\beta_0)$. Then we have 
\[[\varphi_*\circ \mathscr{I}\circ \hat{\alpha}]=\mathscr{I}(\alpha_1)\cdot \beta\cdot \mathscr{I}(\alpha_2)\cdot \beta^{-1}\in \pi_{1}\Emb_{\dagger}(I,I\times Y).\]
\end{Corollary}
\begin{proof}
This follows from  (\ref{eq: decomposition of alpha}) and Proposition \ref{prop: calculation of alpha}.   
\end{proof}

To proceed, we need the following lemma. Recall that $\alpha = \alpha_1\alpha_2\in\pi_1(Y)$. Also recall that $\pi_1(Y)$ acts on  $\mathbb{Z}[\pi_{1}(Y)\setminus \{1\}]$ by conjugations on generators. For each $\hat\alpha\in \pi_1(Y)$, let 
\[
C_{\hat\alpha}:\mathbb{Z}[\pi_{1}(Y)\setminus \{1\}]\to \mathbb{Z}[\pi_{1}(Y)\setminus \{1\}]
\]
denote the conjugation action of $\hat\alpha$.

\begin{Lemma}\label{lem: beta_1 doesn't exist}
There does not exist $\beta_{1}\in \mathbb{Z}[\pi_{1}(Y)\setminus \{1\}]$ such that \[\beta_1 - C_{\alpha_2\alpha_1}(\beta_1) = \beta_0 - C_{\alpha_2}(\beta_0).\]   
\end{Lemma}
\begin{proof} 
We have 
\[
\begin{split}
C_{\alpha_2}(\beta_{0})&=\sum^{n}_{k=1}c_{k}(\alpha_{2}(\alpha_{1}\alpha_{2})^{k}\alpha_{2}^{-1}+\alpha_{2}(\alpha^{-1}_{2}\alpha^{-1}_{1})^k\alpha_{2}^{-1})\\&=\sum^{n}_{k=1}c_{k}((\alpha_{2}\alpha_{1})^{k}+(\alpha^{-1}_{1}\alpha^{-1}_{2})^k)
\end{split}
\]
The coefficient of $(\alpha_2\alpha_1)^{k}$ in $C_{\alpha_2}(\beta_{0})$ is $c_k$. Since $\alpha_2$ is not $2$--torsion, the coefficient of $(\alpha_2\alpha_1)^{k}$ in $\beta_0$ is $0$. Hence the coefficient of  $(\alpha_2\alpha_1)^{k}$ in $\beta_0 - C_{\alpha_2}(\beta_0)$ is $-c_k$. On the other hand, the coefficient of $(\alpha_2\alpha_1)^{k}$ in $\beta_1 - C_{\alpha_2\alpha_1}(\beta_1)$ is always zero because $\alpha_{2}\alpha_{1}$ commutes with $(\alpha_2\alpha_1)^{k}$. Since $(c_k)\neq 0\in\mathbb{Z}^n$, we have $\beta_1 - C_{\alpha_2\alpha_1}(\beta_1) \neq \beta_0 - C_{\alpha_2}(\beta_0)$ for every $\beta_1$.
\end{proof}

Now we can finish the proof of Proposition \ref{prop: loop nonisotopic}.
\begin{proof}[Proof of Proposition \ref{prop: loop nonisotopic}] 
Let $\beta=F(\beta_0)$.
It suffices to show that the elements
\[[\mathscr{I}\circ \hat{\alpha}]=\mathscr{I}(\alpha_1)\cdot \mathscr{I}(\alpha_2)\]
and 
\[[\varphi_*\circ \mathscr{I}\circ \hat{\alpha}]=\mathscr{I}(\alpha_1)\cdot \beta\cdot \mathscr{I}(\alpha_2)\cdot \beta^{-1} \] are not conjugate in $\pi_{1}\Emb_{\dagger}(I,I\times Y)$. Suppose this is not the case, then there exists $w\in \pi_{1}\Emb_{\dagger}(I,I\times Y)$ such that 
\begin{equation}\label{eq: conjugation}
w\cdot \mathscr{I}(\alpha_1)\cdot \mathscr{I}(\alpha_2)\cdot w^{-1}=\mathscr{I}(\alpha_1)\cdot \beta\cdot \mathscr{I}(\alpha_2)\cdot \beta^{-1}    
\end{equation}
By Proposition \ref{prop_decomposition_pi_1_emb_dagger}, we have
\[
\pi_1\Emb_\dagger(I,I\times Y)\cong (\mathbb{Z}[\pi_1(Y)\setminus\{1\}]\oplus \pi_2(Y))\rtimes\pi_1(Y),
\]
where $\mathscr{I}:\pi_1(Y)\to \pi_1\Emb_\dagger(I,I\times Y)$ coincides with the inclusion map to the $\pi_1(Y)$ component in the semidirect product. In the following, we will identify $\pi_1(Y)$ with its image by $\mathscr{I}$, and view $\pi_1\Emb_\partial(I,I\times Y)$ as a subgroup of $\pi_1\Emb_\dagger(I,I\times Y)$.
Then we can write $w=x\cdot  y$, where 
\[x\in \pi_{1}\Emb_{\partial}(I,I\times Y)\cong \mathbb{Z}[\pi_1(Y)\setminus\{1\}]\oplus \pi_2(Y),\,\,\,  y\in \pi_{1}(Y).\]
The projection of \eqref{eq: conjugation} to $\pi_1(Y)$ yields $y\alpha y^{-1}=\alpha$. Hence \eqref{eq: conjugation} implies
\[
x\alpha x^{-1}=\alpha_1\beta\alpha_2\beta^{-1},
\]
which can be rewritten as
\[
(\alpha_1^{-1}x\alpha_1)(\alpha_2x^{-1}\alpha_2^{-1}) = \beta(\alpha_2\beta^{-1}\alpha_2^{-1}).
\]
Let $x' = \alpha_1^{-1}x\alpha_1$, then 
\[
x'\big((\alpha_2\alpha_1)(x')^{-1}(\alpha_2\alpha_1)^{-1}\big) = \beta(\alpha_2\beta^{-1}\alpha_2^{-1}).
\]
Let $\beta_1\in \bZ[\pi_1(Y)\setminus\{1\}]$ be the projection of $x$ under the map 
\[
\pi_{1}\Emb_{\partial}(I,I\times Y)\cong \mathbb{Z}[\pi_1(Y)\setminus\{1\}]\oplus \pi_2(Y) \to \mathbb{Z}[\pi_1(Y)\setminus\{1\}].
\]
Then we have
\[
\beta_1 - C_{\alpha_2\alpha_1}(\beta_1) = \beta_0 - C_{\alpha_2}(\beta_0).
\]
By Lemma \ref{lem: beta_1 doesn't exist}, such $\beta_1$ does not exist, so there is no $w$ satisfying \eqref{eq: conjugation}. 
\end{proof}

\bibliographystyle{amsalpha}
\bibliography{references}

\end{document}